\documentclass[11pt, twoside]{article}

\usepackage{geometry}
\usepackage{amsmath, amsthm, amssymb, mathrsfs, upgreek}
\usepackage{graphicx, caption, float, multirow}
\usepackage{chngcntr, etoolbox}

\usepackage[hidelinks]{hyperref}
\usepackage[dvipsnames]{xcolor}
\hypersetup{colorlinks=true, linkcolor=BurntOrange, citecolor=BurntOrange}

\usepackage{fouriernc, anyfontsize}
\usepackage[T1]{fontenc}

\counterwithin*{equation}{section}
\counterwithin*{equation}{subsection}

\allowdisplaybreaks

\usepackage{fancyhdr}
\newtheoremstyle{theorem}
  {10pt}
  {10pt}
  {\sl}
  {}
  {\bf}
  {. }
  { }
  {}
\theoremstyle{theorem}
\newtheorem{theorem}{Theorem}[section]
\newtheorem{assumption}[theorem]{Assumption}
\newtheorem{definition}[theorem]{Definition}
\newtheorem{lemma}[theorem]{Lemma}

\newtheorem{corollary}[theorem]{Corollary}
\newtheorem{remark}[theorem]{Remark}

\theoremstyle{definition}

\theoremstyle{definition}

\begin{document}
\title{\textbf{Characterizing nonconvex boundaries \\ via scalarization}}
\author{Jin Ma~\thanks{Department of Mathematics, University of Southern California. Email: \underline{jinma@usc.edu}. This author is supported in part by \text{U.S.} NSF grant DMS\#2510403.} \and Weixuan Xia~\thanks{Department of Mathematics, University of Southern California. Email: \underline{weixuanx@usc.edu}. Correspondence: 3620 S Vermont Ave, Los Angeles, CA 90089, U.S.} \and Jianfeng Zhang~\thanks{Department of Mathematics, University of Southern California. Email: \underline{jianfenz@usc.edu}. This author is supported in part by \text{U.S.} NSF grant DMS\#2510403.}}
\date{\today}

\maketitle

\thispagestyle{plain}

\begin{abstract}
  We present a unified approach for characterizing the boundary of a possibly nonconvex domain.
  Motivated by the well-known Pascoletti--Serafini method of scalarization, we recast the boundary characterization as a multi-criteria optimization problem with respect to a local partial order induced by a spherical cone with varying orient. Such an approach enables us to trace the whole boundary and can be considered a general dual representation for arbitrary (nonconvex) sets satisfying an exterior cone condition. We prove the equivalence between the geometrical boundary and the scalarization-implied boundary, particularly in the case of Euclidean spaces and two infinite-dimensional spaces for practical interest. By reformulating each scalarized problem as a parameterized constrained optimization problem, we shall develop a corresponding numerical scheme for the proposed approach. Some related applications are also discussed.
\end{abstract}

\newcommand{\dd}{{\rm d}}
\newcommand{\pd}{\partial}
\newcommand{\PP}{\mathbb{P}}
\newcommand{\E}{\mathbb{E}}
\newcommand{\0}{\mathbf{0}}
\newcommand{\Int}{\mathrm{int}}
\newcommand{\cl}{\mathrm{cl}}
\newcommand{\co}{\mathrm{co}}

\medskip

\section{Introduction}\label{sec:1}

Characterizing the boundary of a general (potentially highly nonconvex) domain in a topological space -- by identifying all of its boundary points -- is a fundamental task across various fields, which aids in the analysis of solution existence, stability, efficiency, and generalization. Broadly speaking, understanding a set's boundary provides crucial structural information, particularly in optimization, where it understandably reduces the effective dimensionality of the problem at hand.

The analysis of the evolution of sets (or set-valued functions) and their geometrical properties forms the core of many problems in applied mathematics. Well-known examples include
stochastic viability problems (see [Aubin--Da Prato, 1998] \cite{ADP98}), multivariate super-hedging problems ([Kabanov, 1999] \cite{K99} and [Bouchard--Touzi, 2000] \cite{BT00}), multivariate dynamic risk measures ([Feinstein--Rudloff, 2015] \cite{FR15}), time inconsistent optimization problems ([Karman--Ma--Zhang, 2017] \cite{KMZ17}), stochastic target problems ([Soner--Touzi, 2002, 2003] \cite{ST02, ST03}), nonzero sum games and mean field games (e.g., [Feinstein--Rudloff--Zhang, 2022] \cite{FRZ22} and [\.{I}\c{s}eri--Zhang, 2024a] \cite{IZ24a}), as well as
multi-criteria stochastic control or optimization problems (see, e.g., [\.{I}\c{s}eri--Zhang, 2024b] \cite{IZ24b} and [Xia, 2024] \cite{X24}). The characterization of the boundaries of general geometric domains holds high relevance across various other areas. This includes 
its established role in the study of surface evolution equations (see, e.g., [Sethian, 1985] \cite{S85}, [Evans--Spruck, 1991] \cite{ES91}, [Soner, 1993] \cite{S93}, [Barles--Soner--Souganidis, 1993] \cite{BSS93}, and the monograph [Giga, 2006] \cite{G06} for detailed exposition),
and, more recently, its profound connections to machine learning -- especially for training deep neural networks by uncovering the spatial structure of loss landscapes, or the so-named ``neuro-manifolds'' ([Li et al., 2018] \cite{LXTSG18}, [Calin, 2020] \cite{C20}, and [Marchetti et al., 2025] \cite{MSMTK25});\footnote{These landscapes are often highly complex due to the compositional nature of activation functions, and their analysis primarily focuses on understanding the geometry of the objective space under the neural network map to gain insights into the model training process.} similarly, in constrained reinforcement learning, where the space of feasible policies can be exceedingly nonconvex subject to complex environments, a characterization of the boundary of the policy space has the potential to significantly reduce tentative search regions for training (see, e.g.,
[Hambly--Xu--Yang, 2023] \cite{HXY23} and [Milani et al., 2024] \cite{MTVF24}).

In all of these applications, the set (domain) under examination, whether finite- or infinite-dimensional, generally lacks guaranteed convex properties. The main goal of the present paper is then to provide a representation for the boundary of such a set, leveraging techniques from multi-criteria optimization. Our main result gives rise to a dual representation of nonconvex boundaries with only a minor nonrestrictive exterior cone condition, which is of independent interest and, to the best of our knowledge, is novel.


More specifically, our boundary characterization approach relies on multi-criteria optimization with respect to some partial order induced by a cone (\text{a.k.a.} an ordering cone). In essence, the search of any boundary point of a set leads to a multi-criteria optimization problem within one of its subsets -- as a suitably constrained problem. If the set happens to be convex, this subset necessarily coincides with the set itself, and the classical (weighted-sum) scalarization method of Gass--Saaty ([Gass--Saaty, 1955] \cite{GS55}) implies that the entire boundary can be characterized by considering its support functions.

The Pascoletti--Serafini scalarization method ([Pascoletti--Serafini, 1984] \cite{PS84}), on the other hand, offers the flexibility to handle potential nonconvexity. In a usual fashion, when applied with a fixed ordering cone, it can identify Pareto-optimal points, even those lying on nonconvex frontiers of the objective space. If we further allow the orient of this cone to vary, this method has the potential to exhaust all such boundary points. In the convex case, these ordering cones can simply be taken as half-spaces, making the Pascoletti--Serafini scalarization a natural generalization that sharpens the cone to detect optimal points in nonconvex regions as well.

However, the standard (infinite) ordering cones used in the Pascoletti--Serafini method impose a restriction on the degree of nonconvexity the set can exhibit. Specifically, points on ``too nonconvex'' boundaries can remain undetected -- technically, when the recession cone of the complement is the singleton $\{\0\}$. To give an example in 2D, as Figure \ref{fig:1} shows, the boundary of the left domain can be completely characterized using such standard cones with a varying orient, whereas that of the right one cannot -- in particular, it is not possible to capture the bolded part of the boundary due to necessary intersection with these (standard) cones (such as the orange one). To overcome this limitation, we introduce a localization method by employing spherical cones (which are bounded), so that the optimization problem can be formulated and solved in a strictly local sense. It turns out that such a method enables us to effectively capture nonconvex boundary points without undesirably extending into the enclosed regions.

\begin{figure}[H]
  \centering
  \includegraphics[scale=0.5]{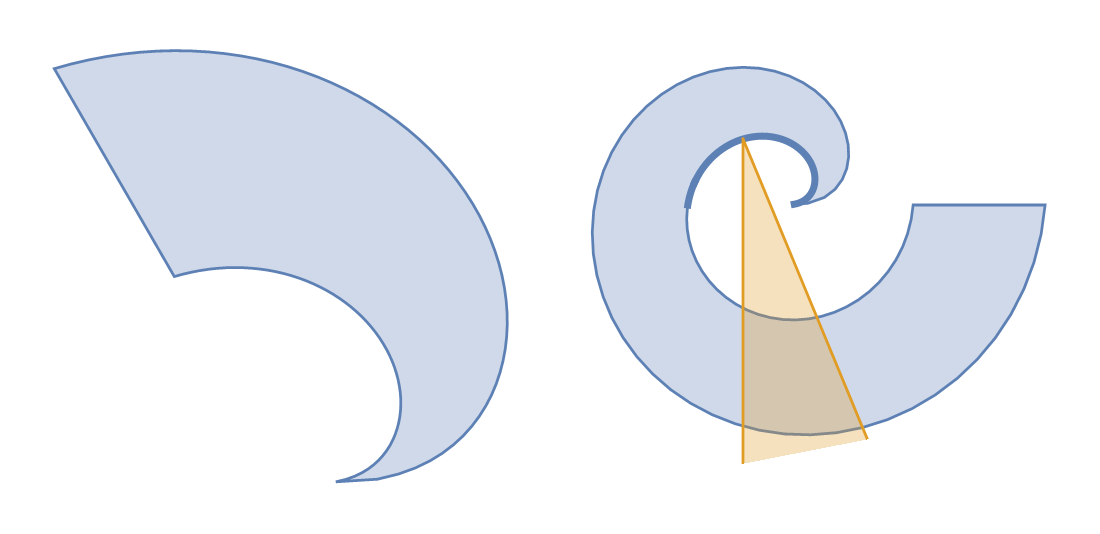}\\
  \caption{Exemplary nonconvex domains in 2D (I)}
  \label{fig:1}
\end{figure}

There is a long strand of literature concerning multi-criteria optimization with a focus on developing scalarization techniques to cope with possible nonconvexity. We refer to the comprehensive surveys [Ehrgott, 2005] \cite{E05}, [Eichfelder, 2008] \cite{E08}, and [Jahn, 2011] \cite{J11} and highlight alternative formulations in a set-valued, lattice-based framework in [Hamel et al., 2015] \cite{HHLRS15}. The main reason to consider the Pascoletti--Serafini type scalarization lies in its overall flexibility: It allows for the choice of an ordering cone and provides control over tradeoff direction via a reference point. It can handle nonconvexity while capturing all Pareto-optimal points, unlike the Gass--Saaty method, which is limited to convex fronts. The Pascoletti--Serafini method is often computationally less complex than the $\varepsilon$-constraint method, which requires iterative adjustments. Compared to Chebyshev scalarization, the Pascoletti--Serafini method operates directly on the natural scale of the objective space and does not require re-scaling or weight tuning. A detailed comparison can be found in [Kasimbeyli et al., 2019] \cite{KOKYE19}.


By providing dual representations of nonconvex boundaries, this work also runs alongside the literature on theories of nonconvex duality. For convex optimization problems, dual problems can be directly formulated with conjugate functions based on the Fenchel--Legendre transformation, and strong duality in the sense of a zero-duality gap can be established under certain stability conditions ([Rockafellar, 1974] \cite{R74}). In the absence of convexity, there have been numerous attempts over the years to extend this framework and establish similar strong duality results -- e.g., [Balder, 1977] \cite{B77}, [Toland, 1978] \cite{T78}, [Mart\'{\i}nez-Legaz and Volle, 1999] \cite{M-LV99}, [Azimov--Gasimov, 2002] \cite{AG02}, and [Sharikov, 2009] \cite{S09}. This extension is also heavily tied to the so-called ``general augmented Lagrangian functions,'' as comprehensively studied in [Rockafellar--Wets, 2009] \cite{RW09} and [Yang--Huang, 2001] \cite{YH01}; see also [Flores-Bazan--Mastroeni, 2013] \cite{F-BM13} and [Yalcin--Kasimbeyli, 2024] \cite{YK24} for recent developments in the context of constrained optimization.

Our results can be viewed as an extension of existing work on nonconvex duality by generalizing the (single-valued) objective function to a multi-valued function, meaning that the objective function can be locally one-to-many, allowing for regionally multiple optimal values. The basic idea is that the graph of every such (multi-valued) objective function defined in $\mathbb{R}^{m}$ can be considered the local boundary of a subset of $\mathbb{R}^{m+1}$. For infinite-dimensional spaces, the same property holds in terms of product spaces; see, e.g., [Borges, 1967, \text{Def.} 2.4] \cite{B67}.


We should point out that our proposed boundary characterization is valid for both finite-dimensional and infinite-dimensional spaces, despite a reasonable emphasis on the former. In particular, in the finite-dimensional setting, we shall demonstrate a specific dimensionality reduction technique within the space of parameters for compact domains, which simplifies the characterization computationally -- a feature that has significant implications
for numerical implementation. In the infinite-dimensional setting, we shall address some key subtleties arising from the lack of reflexivity and show how the characterization remains effective via a parametric link to the topological dual space -- a connection not readily inferred from the finite-dimensional case.

Finally, it is worth mentioning that our results bear important connections to recently noted machine learning applications, especially in constrained reinforcement learning. More specifically, we are able to characterize the boundary of the set of all admissible policies satisfying possibly complex structured constrains of the environment. While this conceptual quantification has only recently been observed, finding the boundary of the policy space is a rather tall task. From what we can gather, our work presents the first such result in the machine learning literature.

The rest of the paper is organized as follows. In Section \ref{sec:2}, we introduce some general notations and basic concepts in geometry, functional analysis, and multi-criteria optimization. In Section \ref{sec:3}, we introduce the scalarization-implied boundary and establish its equivalence to the definitional boundary, in finite-dimensional spaces. In Section \ref{sec:4}, we discuss some dimensionality reduction techniques and reformulate the scalarized problems in a way that standard optimization methods can be applied. Section \ref{sec:5} presents extended results to infinite-dimensional spaces. Finally, in Section \ref{sec:6}, we provide more detailed discussion on related applications of the boundary characterization approach, building on the concepts introduced earlier. Conclusions are drawn in Section \ref{sec:7}.

\medskip

\section{Preliminaries}\label{sec:2}

Consider a Euclidean space $\mathbb{R}^{m}$ ($m\geq1$), for which $\langle\cdot,\cdot\rangle$ and $\|\cdot\|$ denote the inner product and the Euclidean norm, respectively. Let $\mathrm{Cl}(\mathbb{R}^{m})$ denote the space of all nonempty closed subsets of $\mathbb{R}^{m}$ and $\mathrm{Conv}(\mathbb{R}^{m})$ denote the space of all convex subsets in $\mathrm{Cl}(\mathbb{R}^{m})$. For a domain $D\subset\mathbb{R}^{m}$,
we denote by $\cl D$ and $\Int D$ the closure and interior of $D$ in $\mathbb{R}^{m}$, respectively. Then, the \textsl{boundary} of $D$ is defined as $\pd D:=\cl D\setminus\Int D$. Clearly, if
$f\in\pd D$ and $O$ is any neighborhood of $f$, then $D\cap O\neq\emptyset$ and $D\cap O^{\complement}\neq\emptyset$, where $O^{\complement}$ denotes the complement of $O$. Finally, $\overline{B}(r)=\{\beta\in\mathbb{R}^{m}:\|\beta\|\leq r\}$, with $r\in(0,\infty]$, denotes the closed ball centered at $\0$ with radius $r$.

A \textsl{convex cone} $K\subset\mathbb{R}^{m}$ is a subset satisfying that $\alpha(f+g)\in K$ for all $\alpha\geq0$ and $f,g\in K$. It is said to be \textsl{pointed} if $K\cap(-K)=\{\0\}$. More specifically, for a unit vector $\nu\in \overline{B}(1)$ and a constant $\eta\in(0,1]$, we define a \textsl{circular cone} (\text{a.k.a.} ice cream cone) as
\begin{equation}\label{2.1}
  K_{\nu}(\eta):=\{\beta\in\mathbb{R}^{m}:\; (1-\eta)\|\beta\|\leq\langle\beta,\nu\rangle\},
\end{equation}
which has orient (or axis vector) $\nu$ and aperture (or half-angle) equal to $\arccos(1-\eta)$. It is clear that $K_{\nu}(\eta)$ is a pointed convex cone; it is also closed and \textsl{solid} in the sense that $\cl K_{\nu}(\eta)=K_{\nu}(\eta)$ and $\Int K_{\nu}(\eta)\neq\emptyset$, provided $\eta>0$. We then introduce the following definition.
\begin{definition}\label{def:1}
For a circular cone $K_{\nu}(\eta)$ defined in (\ref{2.1}), we define the corresponding \textsl{spherical cone} (or hypercone) with radius $r\in(0,\infty]$ by $K^r_\nu(\eta):=\overline{B}(r)\cap K_{\nu}(\eta)$.
\end{definition}

Obviously, the spherical cone in Definition \ref{def:1} is a solid, truncated convex cone, except for $r=\infty$, as $K_\nu^\infty(\eta)=K_\nu(\eta)$.
Besides, the circular cone $K_{\nu}(\eta)$
is not a polyhedral cone (i.e., generated by a finite set of vectors) for $m\geq3$, see, e.g., 
\cite[Lemma 2.2]{E09}.\footnote{In particular, for $m\geq3$, (\ref{2.1}) does not include the positive cone $\mathbb{R}^{m}_{+}$ as a special case.}

Given a circular cone $K_{\nu}(\eta)$, we  define a partial order $\succeq\equiv\succeq_{K_{\nu}(\eta)}$ between vectors $f,g\in D$, denoted by $f\succeq g$ (or $f\succeq_{K_{\nu}(\eta)}g$), whenever $f-g\in K_{\nu}(\eta)$.
Accordingly,  we can introduce a ``local partial order'', denoted by $\succeq_r\equiv \succeq_{K^r_{\nu}(\eta)}$, $r\in(0,\infty]$, by a slight abuse of notation. Clearly, for any $f,g\in D$, $f\succeq_r g$ means that $f\succeq g$
and $\|f-g\|\leq r$.

Most often, if we think of these vectors as objective function values from some (possibly stochastic) optimization problem with control variable $x\in\mathcal{A}$, for a known admissibility set $\mathcal{A}$ of a certain function space, with $D=\cl\{f(x):\; x\in\mathcal{A}\}$, then the partial order can be naturally associated with a multi-criteria optimization problem to
\begin{equation}
\label{2.2}
  \sup_{x\in\mathcal{A}}f(x)\equiv\sup_{x\in\mathcal{A}}(f_{1}(x),\dots,f_{m}(x))^{\intercal}
\end{equation}
with respect to the local partial order $\succeq_r$ for some $r\in(0,\infty]$. The notion of maximality with respect to $\succeq_r$ is formalized below.

\begin{definition}\label{def:2}
For the spherical cone $K^r_{\nu}(\eta)$ in Definition \ref{def:1}, a control $x^{\ast}\in\mathcal{A}$ is said to be a (weakly) $K^r_{\nu}(\eta)$-maximal solution to the problem (\ref{2.2}) if $(f(x^{\ast})+\breve{K}^r_{\nu}(\eta))\cap D=\emptyset$, where $\breve{K}^r_{\nu}(\eta):=\overline{B}(r)\cap\Int K_{\nu}(\eta)$.
\end{definition}

The weak $K^r_{\nu}(\eta)$-maximality in Definition \ref{def:2} carries a meaning of (constrained) Pareto optimality: i.e., there is no other control $x\in D\setminus\{x^{\ast}\}$ such that $f(x)-f(x^{\ast})\in\breve{K}^r_{\nu}(\eta)$. Note also that $\breve{K}^r_{\nu}(\eta)\neq\Int K^r_{\nu}(\eta)$ in general, and thus the (weak) maximality is governed by $\succeq$ induced by $K_{\nu}(\eta)$ in (\ref{2.1}) in the usual sense -- not by the constraint (or truncation) by $\overline{B}(r)$.\footnote{It is possible to consider strong maximality by replacing this condition with $(f(x^{\ast})+ K^r_{\nu}(\eta))\cap D=\{f(x^{\ast})\}$, but emphasis will be on the weak notion.}

\medskip

\section{Boundary characterization by way of scalarization}\label{sec:3}

For every spherical cone based on Definition \ref{def:1}, the maximal solutions to (\ref{2.2}) can be recovered via scalarization of the objective function $f:\mathcal{A}\mapsto\mathbb{R}^{m}$, with $D=\cl\{f(x):\; x\in\mathcal{A}\}\in\mathrm{Cl}(\mathbb{R}^{m})$. We make the following regularity assumption on the boundary $\pd D$ throughout this section.
\begin{assumption}\label{as:1}
There exist constants $r>0$ and $\eta\in(0,1]$ such that for every $f\in\pd D$,
\begin{equation}\label{3.1}
  (f+\breve{K}^r_{\nu}(\eta))\cap D=\emptyset,
\end{equation}
for some $\nu\equiv\nu_{f}\in\pd\overline{B}(1)$, where $\pd\overline{B}(1)=\{\beta\in\mathbb{R}^{m}:\;\|\beta\|=1\}$ is the unit sphere, and $\breve{K}^r_{\nu}(\eta)$ is as given in Definition \ref{def:2}.
\end{assumption}

Condition (\ref{3.1}) is an exterior cone condition; see, e.g.,
\cite[\S4]{CS05}. It means that for any boundary point $f$, there is a (sufficiently small, partially open) spherical cone $\breve{K}^r_{\nu}(\eta)$ and a neighborhood $O\ni f$ such that the shifted cone $f+\breve{K}^r_{\nu}(\eta)$ is disjoint from $D$ but contained in $O$. Such a condition is notably nonrestrictive as it allows $D$ to be generally nonconvex, infinite, or even disconnected. In particular, if $D\in\mathrm{Cl}(\mathbb{R}^{m})$ is convex, then (\ref{3.1}) automatically holds with any $r>0$ and $\eta\in(0,1]$. Every Lipschitz (including smooth) domain $D$ satisfies (\ref{3.1}) (see, e.g.,
\cite[\text{Def.} 8 \& \text{Def.} 9]{BBK13}) -- and so does every (not necessarily convex) polytope. Note also that the choice of the unit vector $\nu\equiv\nu_{f}$ generally depends on the boundary point $f\in\pd D$ and can vary within every neighborhood $O\ni f$, meaning that this cone condition need \textsl{not} be uniform.\footnote{This is in contrast to the uniform interior cone condition considered in
\cite[\text{Condition} ($B'$)]{S87} concerning Skorohod equations; see also
\cite[\text{Def.} 8]{BBK13}.} What Assumption \ref{as:1} specifically disallows are nonconvex sets which have singularities formed by the epigraph of non-Lipschitz functions. Plainly speaking, it prevents any inward-pointing corners on $\pd D$ from being overly sharp. As another example, Figure \ref{fig:2} below shows three nonconvex domains for $m=2$, of which the left two satisfy Assumption \ref{as:1} but the rightmost one does not (due to the cusp singularity). Moreover, it is obvious that more irregular boundaries are associated with smaller (necessary) values of $r$ and $\eta$.

\begin{figure}[H]
  \centering
  \includegraphics[scale=0.5]{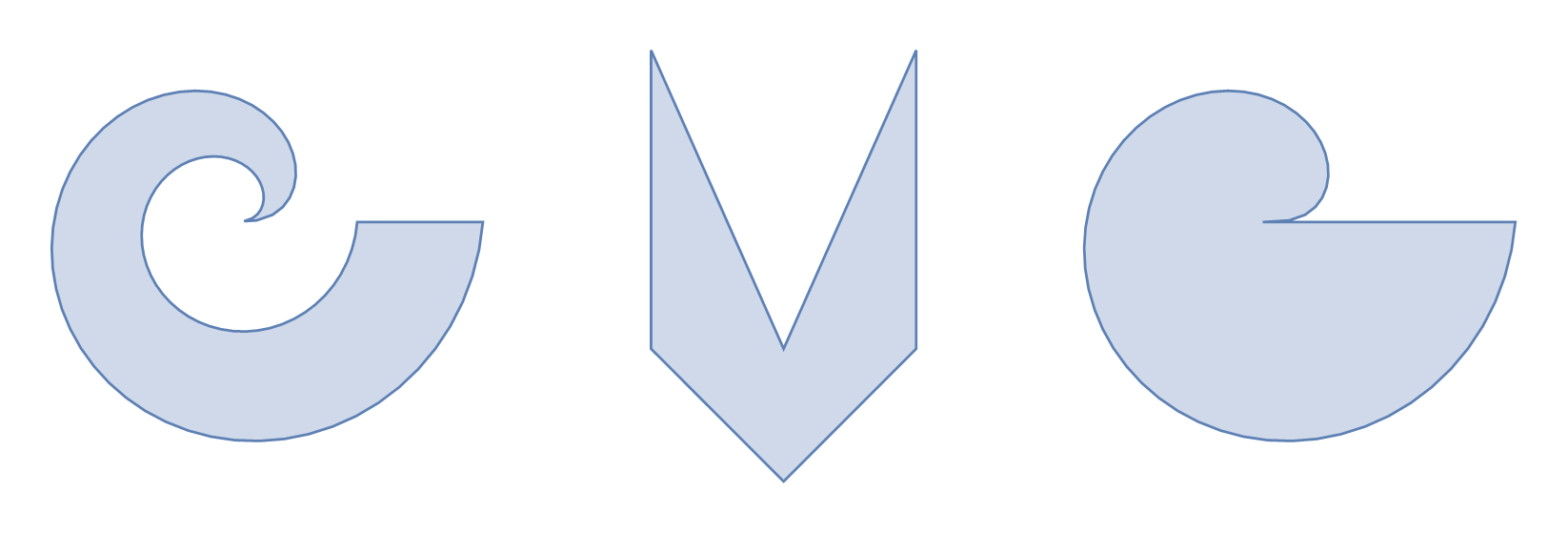}\\
  \caption{Exemplary nonconvex domains in 2D (II)}
  \label{fig:2}
\end{figure}

In order to see how to recover the boundary $\pd D$ by way of scalarization,
we start by outlining the basic intuition via a geometric illustration in 2D, as shown in Figure \ref{fig:3}, which is based on the leftmost domain ($D$) from Figure \ref{fig:2}.
Here, the domain $D$ satisfies condition (\ref{3.1}) with $r=1$ and $\eta=1-\cos(\pi/8)$, and the orange sectors represent the same (closed) spherical cone $K^{1/3}_{b}(1-\cos(\pi/8))$ with fixed orient $b=(-1/\sqrt{2},1/\sqrt{2})^{\intercal}$. Under the standard Pascoletti--Serafini method (see \cite[\S3.5]{KOKYE19}, \cite{PS84}), the standard (infinite) ordering cone $K_{b}(1-\cos(\pi/8))$ (without truncation) is to be shifted -- starting from a point in $D$ -- in the direction of $b$ (along the dashed ray) by a maximal distance until it is out of $D$. Notably, from Figure \ref{fig:3}, while the outer boundary is readily identifiable in this fashion, detecting the inner boundary requires constraining the maximal distance to prevent it from converging directly to the outer one. Meanwhile, it is clear that the spherical cone $K^{1/3}_{b}(1-\cos(\pi/8))$ should be used in place of $K_{b}(1-\cos(\pi/8))$ because the latter (infinite), when placed at the inner boundary, will necessarily intersect $D$ (see Figure \ref{fig:1}). However, imposing this constraint alone introduces a potential side effect: If the constraint is binding and when the starting point is deeply within $D$, then the maximization is likely to end up in an interior point. Therefore, it is critical to also require that a marginal relaxation of the constraint does not yield a strictly larger maximal value. 

\begin{figure}[H]
  \centering
  \begin{minipage}[c]{0.5\linewidth}
  \includegraphics[scale=0.3]{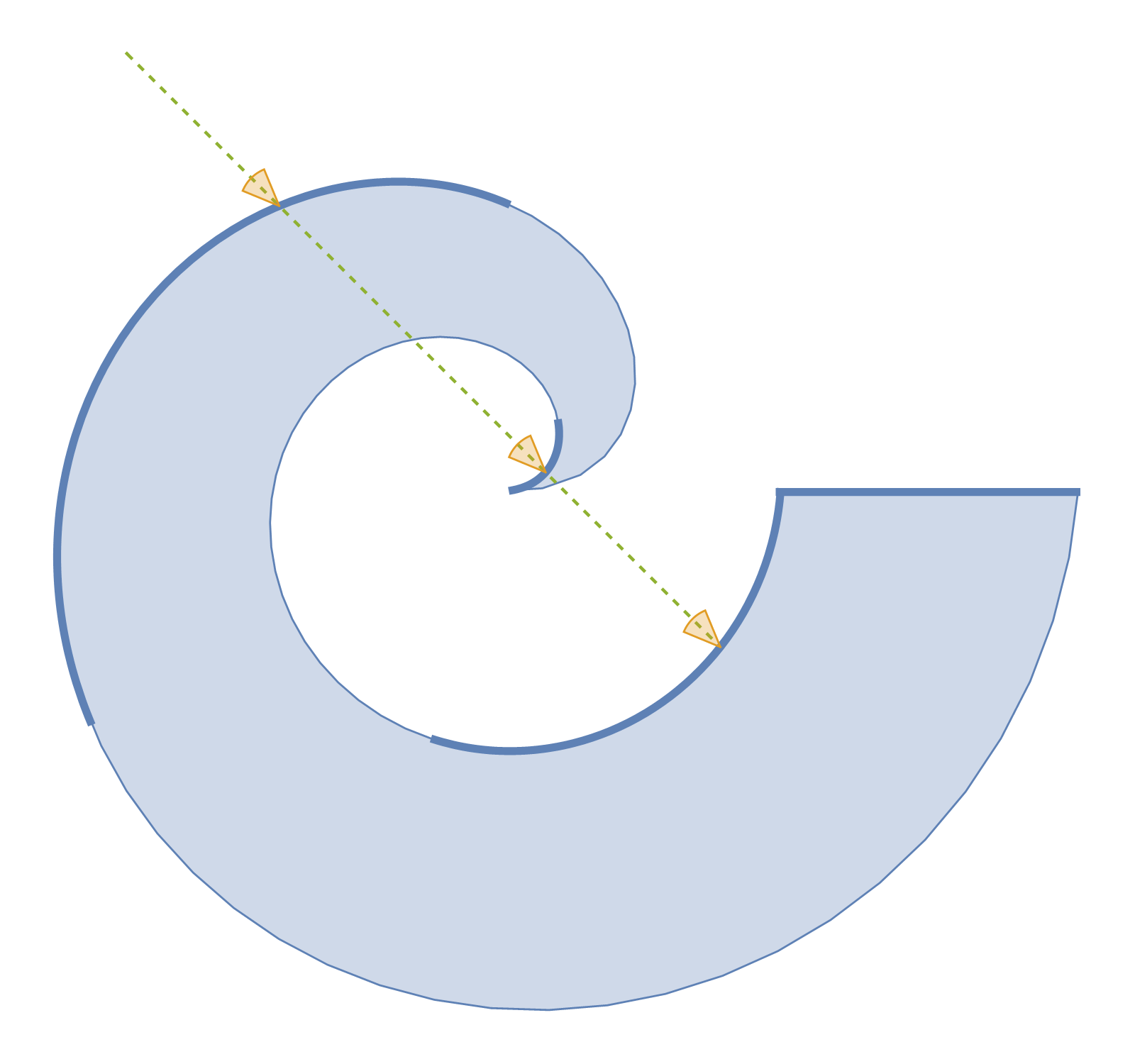}
  \end{minipage}
  \begin{minipage}[r]{0.2\linewidth}
  \includegraphics[scale=0.35]{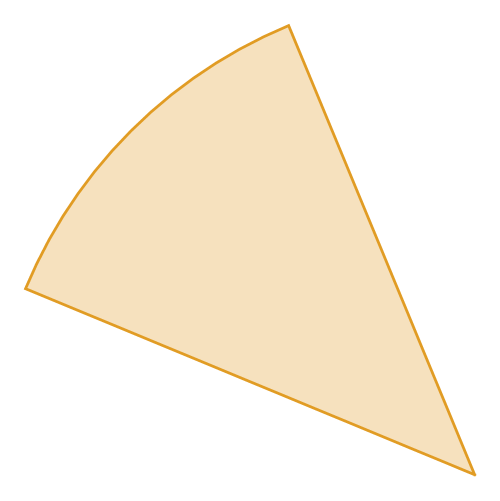}
  \caption*{\small (spherical cone)}
  \end{minipage}
  \caption{Illustration of boundary characterization via scalarization}
  \label{fig:3}
\end{figure}

We shall formalize the above intuition by introducing a sequence of nonlinear \textsl{scalarization functions} of Pascoletti--Serafini type, together with their related optimization problems. For each $k\in\mathbb{Z}_{++}$ and $f\in D$, we consider the following (leveled) \textsl{scalarization sets}, parameterized by $a\in\mathbb{R}^{m}$ (base, or reference point) and $b\in\pd\overline{B}(1)$ (orient):
\begin{equation}\label{3.3}
  \mathfrak{S}^{(k)}_{a,b}(f)=\{y\in[0,k\epsilon]:\; f-a-yb\in K^{\epsilon}_{b}(\eta)\},\quad f\in D,\;k\in\mathbb{Z}_{++},\;\epsilon=\frac{r}{3},
\end{equation}
where $r>0$ is the constant in Assumption \ref{as:1}. Then, we define the scalarization functions
\begin{equation}\label{3.2}
  \phi^{(k)}_{a,b}(f):=\sup_{y\in\mathfrak{S}^{(k)}_{a,b}(f)}y,\quad k\in\mathbb{Z}_{++},
\end{equation}
and denote the ``value functions'' of the corresponding optimization problems by
\begin{equation}\label{3.5}
  V^{(k)}_{a,b}=\sup_{f\in D}\phi^{(k)}_{a,b}(f),\quad k\in\mathbb{Z}_{++}.
\end{equation}

Next, we introduce the \textsl{scalarization-implied boundary} associated with the optimization problem in (\ref{3.5}):
\begin{equation}\label{3.4}
  \pd_{\rm s}D:=\big\{f\in D:\; \phi^{(k)}_{a,b}(f)=V^{(k)}_{a,b},\;\exists a\in\mathbb{R}^{m},\;b\in\pd\overline{B}(1),\;k\in\mathbb{Z}_{++}\text{ s.t. }V^{(k+1)}_{a,b}=V^{(k)}_{a,b}\big\}.
\end{equation}

\begin{remark}\label{rem:1}
For fixed $a\in\mathbb{R}^{m}$, $b\in\pd\overline{B}(1)$, and $k\in\mathbb{Z}_{++}$, equations (\ref{3.2}) and (\ref{3.5}) indeed constitute a scalarized problem of Pascoletti--Serafini type. It consists of moving the (spherical) cone $\breve{K}^\epsilon_{b}(\eta)$ along the line segment $\{f-a-yb:\; y\in[0,k\epsilon]\}$ in the direction $b$ by a maximum distance $y^*=y^*(f)$, then maximizing $y^*$ by varying $f\in D$. For this purpose, the spherical cone $\breve{K}^{r}_{\nu}(v)$ in condition (\ref{3.1}) uses the closed ball $\overline{B}(r)$ (hence partially open) to exclude cases where $(f+\breve{K}^r_{\nu}(\eta))\cap D=\{g\}\subset\pd D$ with $g\neq f$. Then, in constructing (\ref{3.4}), the constraint $V^{(k+1)}_{a,b}=V^{(k)}_{a,b}$ precisely reflects the idea that for any boundary point $f\in\pd D$, extending the above line segment by a marginal length of $\epsilon$ does not lead to a longer possible distance $y$. The choice $\epsilon=r/3$ in defining the scalarization sets in (\ref{3.3}) is made for technical reasons and will be justified later (in step 1 in the proof of Theorem \ref{thm:1}). 
\qed
\end{remark}


Figure \ref{fig:2} also helps illustrate (\ref{3.4}): Using only the (fixed) spherical cone $K^{1/3}_{b}(1-\cos(\pi/8))$, while varying the base $a\in\mathbb{R}^{2}$, one can recover the bolded part of the boundary. On the contrary, without using the spherical cone or the level index $k\in\mathbb{Z}_{++}$, a good part of the inner boundary can never be found.


The main result of this paper is the following theorem regarding the equivalence of (\ref{3.5}) to the original boundary.

\begin{theorem}\label{thm:1}
Under Assumption \ref{as:1}, it holds that $\pd D=\pd_{\rm s}D$.
\end{theorem}

\begin{proof}
\textbf{Step 1.} First, let $f^{\circ}\in\pd D$. By Assumption \ref{as:1}, there exists a (not necessarily unique) unit vector $\nu_{f^{\circ}}$ depending on $f^{\circ}$ such that $(f^{\circ}+\breve{K}^{3\epsilon}_{\nu_{f^{\circ}}}(\eta))\cap D=\emptyset$. Let $a=f^{\circ}$ and $b=\nu_{f^{\circ}}$, and take $k=1$. Noting that the spherical cone $\breve{K}^{3\epsilon}_{\nu_{f^{\circ}}}(\eta)$ is pointed, then by (\ref{3.2}) we have
\begin{equation*}
  \mathfrak{S}^{(1)}_{f^{\circ},\nu_{f^{\circ}}}(f^{\circ})=\{y\in[0,\epsilon]:-y\nu_{f^{\circ}}\in  K^\epsilon_{\nu_{f^{\circ}}}(\eta)\}=\{0\},
\end{equation*}
which implies that $\phi^{(1)}_{f^{\circ},\nu_{f^{\circ}}}(f^{\circ})=\sup_{y\in\mathfrak{S}^{(1)}_{f^{\circ},\nu_{f^{\circ}}}(f^{\circ})}y=0$. Thus, by (\ref{3.4}),
\begin{equation}\label{3.6}
  V^{(1)}_{f^{\circ},\nu_{f^{\circ}}}=\sup_{f\in D}\phi^{(1)}_{f^{\circ},\nu_{f^{\circ}}}(f)\geq\phi^{(1)}_{f^{\circ},\nu_{f^{\circ}}}(f^{\circ})=0.
\end{equation}
Also, since $f^{\circ}\in\pd D$, by condition (\ref{3.1}) again we have that $f-f^{\circ}\notin\breve{K}^{3\epsilon}_{\nu_{f^{\circ}}}(\eta)$ for any $f\in D$, while by (\ref{3.3}), $f-f^{\circ}\in y\nu_{f^{\circ}}+K^{\epsilon}_{\nu_{f^{\circ}}}(\eta)$ for any $f\in D$ and $y\in\mathfrak{S}^{(1)}_{f^{\circ},\nu_{f^{\circ}}}(f)$, but since $\nu_{f^{\circ}}=b$ is the orient of $K_{\nu_{f^{\circ}}}(\eta)$, and $\overline{B}(\epsilon)+y\nu_{f^{\circ}}\subset\overline{B}(3\epsilon)$ for any $y\in[0,\epsilon]$, it implies that $y=0$ whenever it exists, and so $\mathfrak{S}^{(1)}_{f^{\circ},\nu_{f^{\circ}}}(f)\subset\{0\}$ for all $f\in D$, and it follows from (\ref{3.6}) that
\begin{equation*}
  V^{(1)}_{f^{\circ},\nu_{f^{\circ}}}=0=\phi^{(1)}_{f^{\circ},\nu_{f^{\circ}}}(f),\quad\forall f\in D.
\end{equation*}

To verify the constraint in (\ref{3.5}), take $k=2$, and arguing similar as above we have
\begin{equation*}
  \mathfrak{S}^{(2)}_{f^{\circ},\nu_{f^{\circ}}}(f^{\circ})=\{y\in[0,2\epsilon]:-y\nu_{f^{\circ}}\in K^\epsilon_{\nu_{f^{\circ}}}(\eta)\}=\{0\},
\end{equation*}
and since $f-f^{\circ}\notin\breve{K}^{3\epsilon}_{\nu_{f^{\circ}}}(\eta)$ and $\overline{B}(\epsilon)+yb\subset\overline{B}(3\epsilon)$ for any $y\in[0,2\epsilon]$, we again conclude that $V^{(2)}_{f^{\circ},\nu_{f^{\circ}}}=0$. From the definitional condition (\ref{3.5}), it follows that $\pd D\subseteq\pd_{\rm s}D$.

\textbf{Step 2.} Conversely, let $f^{\circ}\in\pd_{\rm s}D$. Noting that $\mathfrak{S}^{(k)}_{a,b}(f^{\circ})\subset\mathbb{R}_{+}$ is compact by construction, by condition (\ref{3.5}) there exist $a\in\mathbb{R}^{m}$, $b\in\pd\overline{B}(1)$, and $k\in\mathbb{Z}_{++}$ such that $V^{(k)}_{a,b}=V^{(k+1)}_{a,b}$ and $\phi^{(k)}_{a,b}(f^{\circ})=V^{(k)}_{a,b}:=y^{\circ}\in[0,k\epsilon]$. We want to show that $f^{\circ}\in\pd D$.

Observe that if condition (\ref{3.1}) for $\pd D$ holds for a given ball radius $r=3\epsilon>0$, then it is evident to hold for any smaller radius within $(0,3\epsilon)$.\footnote{Recall that for any boundary point $f\in\pd D$, $(f+\overline{B}(\varepsilon))\setminus D\neq\emptyset$, $\forall\varepsilon>0$; see also 
\cite[\S2]{BBK13}.} In what follows we specify such a (smaller) radius to $\epsilon/2$. By way of contradiction, suppose that $f^{\circ}\notin\pd D$, meaning that for any $\nu\in\pd\overline{B}(1)$, there exists $\hat{f}(\nu)\in D$ such that $\hat{f}(\nu)\in f^{\circ}+\breve{K}^{\epsilon/2}_{\nu}(\eta)$, or equivalently, that $\hat{f}(\nu)=f^{\circ}+\hat{\beta}$ for some $\hat{\beta}\in\breve{K}^{\epsilon/2}_{\nu}(\eta)$. Fix such a vector $\nu=b$.

Then, based on $\mathfrak{S}^{(k)}_{a,b}(f^{\circ})$ from (\ref{3.3}), we can define $\beta^{\circ}:=f^{\circ}-a-y^{\circ}b\in K^\epsilon_{b}(\eta)$. In fact, we can assume without loss of generality that $f^{\circ}-a-y^{\circ}b\in K^{\epsilon/2}_{b}(\eta)$, for otherwise we can work with a new parameter $a'=a+\tilde{a}$ for some $\tilde{a}\in K^\epsilon_{b}(\eta)$ and repeat the same argument. In any case, we have that $\hat{f}(b)-a-y^{\circ}b=\hat{\beta}+\beta^{\circ}\in\breve{K}^\epsilon_{b}(\eta)$.

Now, since $\Int K_{b}(\eta)\neq\emptyset$, there exists $\varepsilon\in(0,\epsilon)$ such that $\hat{\beta}+\beta^{\circ}-\varepsilon b\in\breve{K}^\epsilon_{b}(\eta)$. Hence, by assigning $\hat{y}:=y^{\circ}+\varepsilon$, we have that $\hat{f}(b)-a-\hat{y}b\in \breve{K}^\epsilon_{b}(\eta)$. According to the definitional condition (\ref{3.5}), if $\hat{y}\leq k\epsilon$, it follows that $\hat{y}\in\mathfrak{S}^{(k)}_{a,b}(\hat{f})$ but $\hat{y}>y^{\circ}$, contradicting the fact that $V^{(k)}_{a,b}=y^{\circ}$. On the other hand, suppose $\hat{y}>k\epsilon$, or more precisely, $\hat{y}\in(k\epsilon,(k+1)\epsilon)$ because $\varepsilon<\epsilon$; then, since $\hat{f}(b)\in D$, we have that $V^{(k+1)}_{a,b}\geq\hat{y}>y^{\circ}=V^{(k)}_{a,b}$, hence $V^{(k+1)}_{a,b}\neq V^{(k)}_{a,b}$, violating the constraint in (\ref{3.5}) -- again a contradiction. Thus, $f^\circ \in\pd D$, or $\pd_{\rm s} D\subseteq\pd D$, whence
$\pd D=\pd_{\rm s}D$, proving the theorem.
\end{proof}


The above nonlinear scalarization method can be simplified drastically if the set $D\in\mathrm{Cl}(\mathbb{R}^{m})$ is convex. As aforementioned, in this case condition (\ref{3.1}) holds with arbitrary $r>0$ and $\eta\in(0,1]$. Letting $r\to\infty$ and $\eta=1$, the (partially open) spherical cone in (\ref{3.1}) will be replaced by the open half-space $\Int K_{\nu}(1)$. By taking $\epsilon\to\infty$ accordingly, (\ref{3.3}) reduces to a single scalarization set $\mathfrak{S}^{(1)}_{a,b}(f)=\{y\geq0:f-a-yb\in K_{b}(1)\}$, independent of $k\in\mathbb{Z}_{++}$. To wit, there is only one scalarization functional (and scalarized problem) to consider in (\ref{3.2}) (and (\ref{3.4})), and the implied boundary simplifies into
\begin{equation}\label{3.7}
  \pd_{\rm s}D:=\big\{f\in D:\; \phi^{(1)}_{a,b}(f)=V^{(1)}_{a,b},\;\exists a\in\mathbb{R}^{m},\;b\in\pd\overline{B}(1)\big\}.
\end{equation}
For any given $f\in\pd D$, if we take the orient $b$ as an outward-pointing normal vector at $f$, then $\pd K_{b}(1)=\{\beta\in\mathbb{R}^{m}:\;\langle\beta,b\rangle=0\}$ is a supporting hyperplane containing $f$. Thus, in the case of convex domains, (\ref{3.7}) recovers what the \textsl{supporting hyperplane theorem} implies.

We should also note that it is possible for a nonconvex domain $D$ (such as a bean shape) to satisfy condition (\ref{3.1}) with $r=\infty$ (although $\eta=1$ will be infeasible) -- with no truncation by the closed ball needed. In such a case, a single scalarization function in (\ref{3.2}) (with $k=1$) would suffice.

In general, Theorem \ref{thm:1} gives a \textsl{dual representation} of any (finite-dimensional) nonconvex boundary, which only needs to satisfy the exterior cone condition (\ref{3.1}). As discussed earlier, such a representation can be considered an extension of existing duality results in nonconvex optimization to multi-valued objective functions -- by establishing a strong duality relation between subsets of $\pd D$ and the scalarized problems in (\ref{3.4}); see again Figure \ref{fig:3} and compare, e.g., the duality theorems in
\cite[\text{Thm.} 11.59]{RW09} and 
\cite[\text{Thm.} 9]{AG02}.


\medskip

\section{Further analysis of boundary characterization}\label{sec:4}

In this section we further analyze  the proposed boundary characterization, with the aim of developing techniques to facilitate numerical implementation. We shall address two key aspects, one concerned with the possibility of restricting the definitional parameter space in (\ref{3.2}) and the other serving to reformulate the scalarized problems in (\ref{3.4}) into a standard optimization framework.

\smallskip

\subsection{Dimensionality reduction}\label{sec:4.1}

In many applications, especially when the domain $D$ is bounded (or compact), it is possible to work with (\ref{3.2}) in reduced dimensions. Such \textsl{dimensionality reduction} have important and highly nontrivial implications in numerical implementation. The key idea is to leverage information about the boundary values of $D$ (finite due to boundedness) along each dimension to determine the necessary parametric range ensuring that the proof of Theorem \ref{thm:1} remains valid. More precisely, given each orient parameter $b\in \pd\overline{B}(1)$ in (\ref{3.5}), one can choose the base parameter $a$ appropriately in a (strict) subset of $\mathbb{R}^{m}$, instead of exhausting the whole space. The detailed procedure is presented in Theorem \ref{thm:2} below.

\begin{theorem}\label{thm:2}
Let $m\geq2$ and assume that $D\in\mathrm{Cl}(\mathbb{R}^{m})$ is bounded. For any given $\iota\in\mathbb{Z}\cap[1,m]$, define the reduced boundary
\begin{equation}\label{3.8}
  \breve{\pd}_{\rm s}D(\iota):=\big\{f\in D:\; \phi^{(k)}_{a,b}(f)=V^{(k)}_{a,b},\;\exists a\in P_{b}(\iota),\;b\in\pd\overline{B}(1),\;k\in\mathbb{Z}_{++}\text{ s.t. }V^{(k+1)}_{a,b}=V^{(k)}_{a,b}\big\},
\end{equation}
where $V^{(k)}_{a,b}$ is given by (\ref{3.4}), and
\begin{align}\label{3.9}
  P_{b}(\iota)&=\Bigg\{a\in\mathbb{R}^{m}:\;
  a_{\iota}=
  \begin{cases}
    \underline{f}_{\iota},&\quad\text{if }b_{\iota}\geq0 \\
    \overline{f}_{\iota},&\quad\text{if }b_{\iota}<0,
  \end{cases}
  ; \nonumber\\
  &\quad\;a_{i}\in \big[\underline{f}_{i}-\max\{|\underline{f}_{\iota}|b_{i},|\overline{f}_{\iota}|b_{i}\},\overline{f}_{i} -\min\{|\underline{f}_{\iota}|b_{i},|\overline{f}_{\iota}|b_{i}\}\big],\;\forall i\neq\iota\Bigg\},
\end{align}
where $\underline{f}_{i}=\inf D_{i}\in\mathbb{R}$ and $\overline{f}_{i}=\sup D_{i}\in\mathbb{R}$, with $D_{i}=\{f_{i}:\;f\in D\}$, $i\in\mathbb{Z}\cap[1,m]$. Then, it holds that $\pd D=\breve{\pd}_{\rm s}D(\iota)$.
\end{theorem}

\begin{proof}
Let $\iota\in\mathbb{Z}\cap[1,m]$ be fixed. First, since $P_{b}(\iota)\subset\mathbb{R}^{m}$, by Theorem \ref{thm:1} we have that $\pd D\supset\breve{\pd}_{\rm s}D(\iota)$.

To show the other direction, let $f^{\circ}\in\pd D$, and clearly $f_{i}\in[\underline{f}_i,\overline{f}_{i}]$, $i\neq\iota$. Again, Assumption \ref{as:1} implies the existence of a unit vector $\nu_{f^{\circ}}$ such that $(f^{\circ}+\breve{K}^{3\epsilon}_{\nu_{f^{\circ}}}(\eta))\cap D=\emptyset$. Let us take $b=\nu_{f^{\circ}}$ and engineer $a$ such that
\begin{equation}\label{3.10}
  a_{i}=
  \begin{cases}
    f^{\circ}_{i}-|f^{\circ}_{\iota}|b_{i},&\quad\text{if }i\neq\iota\\
    \begin{cases}
    \underline{f}_{\iota},&\quad\text{if }b_{\iota}\geq0 \\
    \overline{f}_{\iota},&\quad\text{if }b_{\iota}<0,
  \end{cases},&\quad\text{if }i=\iota,
  \end{cases}
  \quad i\in\mathbb{Z}\cap[1,m],
\end{equation}
satisfying the first condition in $P_{b}(\iota)$.

Suppose $b_{\iota}\neq0$. Similarly as in the proof of Theorem \ref{thm:1}, since $f^{\circ}\in\pd D$, condition (\ref{3.1}) implies that $f-f^{\circ}\notin\breve{K}^{3\epsilon}_{b}(\eta)$, $\forall f\in D$; on the other hand, by (\ref{3.2}) we know that for any $f\in D$, $k\in\mathbb{Z}_{++}$, and $y\in\mathfrak{S}^{(k)}_{a,b}(f)$, $f-a-yb\in K^\epsilon_{b}(\eta)$, for which by (\ref{3.10}),
\begin{equation}\label{3.11}
  f_{i}-a_{i}-yb_{i}=
  \begin{cases}
    \displaystyle f_{i}-f^{\circ}_{i}+(|f^{\circ}_{\iota}|-y)b_{i},&\quad\text{if }i\neq\iota \\
    \displaystyle f_{\iota}-f^{\circ}_{\iota}+\bigg(\frac{f^{\circ}_{\iota}-a_{\iota}}{b_{\iota}}-y\bigg)b_{\iota},&\quad\text{if }i=\iota,
  \end{cases}
  \quad i\in\mathbb{Z}\cap[1,m].
\end{equation}
Since $\overline{B}(\epsilon)+yb\subset\overline{B}(3\epsilon)$ for any $y\in[0,2\epsilon]$, it is implied that
\begin{equation*}
  y\leq\min\bigg\{|f^{\circ}_{\iota}|,\frac{f^{\circ}_{\iota}-a_{\iota}}{b_{\iota}}\bigg\}:=\bar{y}_{\iota}.
\end{equation*}
Clearly, under (\ref{3.10}), $\bar{y}_{\iota}\geq0$, and we can choose $k$ so large that $\bar{y}_{\iota}\in[(k-1)\epsilon,k\epsilon]$, which, along with the fact that $\mathfrak{S}^{(k)}_{a,b}(f^{\circ})\subset[0,\bar{y}_{\iota}]$ (due to (\ref{3.11})), yields
\begin{equation*}
  V^{(k)}_{a,b}=\phi^{(k)}_{a,b}(f^{\circ})=\bar{y}_{\iota},
\end{equation*}
but as $0\leq\bar{y}_{\iota}\leq k\epsilon$, we have that $V^{(k+1)}_{a,b}=f^{\circ}_{\iota}$ as well, which verifies the constraint in (\ref{3.5}).

Suppose $b_{\iota}=0$, and then the second piece on the right side of (\ref{3.11}) is understood to be equal to $f_{\iota}-a_{\iota}$, independent of $y$. Thus, we can simply set $a_{\iota}=\underline{f}_{\iota}$ (as in (\ref{3.9})), and it suffices to consider the first piece on the right side of (\ref{3.11}), along with $\bar{y}_{\iota}=|f^{\circ}_{\iota}|$, and the same conclusion is reached.

To complete the proof, observe that the domains of $a_{i}$, for $i\neq\iota$, are precisely as stated in the theorem, where the maximum and minimum depend on the sign of $b_{i}$ for $i\neq\iota$.
\end{proof}

\begin{remark}\label{rem:3}
The piecewise structure for $a$ in (\ref{3.10}) is needed due to $b$ being the orient of the cone $K_{b}(\eta)$, and there is no guarantee of the components of $b$ being nonzero. This feature is notably different from the case of (standard) multi-criteria optimization problems with respect to a fixed cone and a variable orient parameter (compare 
\cite[\text{Thm.} 3.5]{E09}).
\qed
\end{remark}

\begin{remark}\label{rem:4}
In constructing the parametric range $P_{b}(\iota)$ in (\ref{3.9}), the requirement on $a_{\iota}$ is for simplicity, and any value satisfying that $a_{\iota}\leq \underline{f}_{\iota}$ if $b_{\iota}\geq0$ and $a_{\iota}\geq \overline{f}_{\iota}$ if $b_{\iota}<0$ is workable for the same purpose; also, the intervals governing each $a_{i}$, $i\neq\iota$, are not tight in that strictly narrower intervals are possible, but such intervals are arguably the easiest to construct as they are ``dimensionally marginal,'' with $\underline{f}_{i}$ and $\overline{f}_{i}$ focusing on the $i$th dimension, rather than involving other dimensions depending on $b$.
\qed
\end{remark}

Under suitable conditions, Theorem \ref{thm:2} can be applied indefinitely for continued dimensionality reduction, which we state in the next corollary.

\begin{corollary}\label{cor:1}
In the setting of Theorem \ref{thm:2}, assume further that there exists a subset $I\subset\mathbb{Z}\cap[1,m]$ with $\mathrm{card}I\geq2$ such that for every $b\in\overline{B}(1)$ and any $\iota,\iota'\in I$ with $\iota\neq\iota'$,
\begin{align}\label{3.12}
  &\underline{f}_{\iota},\overline{f}_{\iota} \in\big[\underline{f}_{\iota'}-\max\{|\underline{f}_{\iota}|b_{\iota'},|\overline{f}_{\iota}|b_{\iota'}\},\overline{f}_{\iota'} -\min\{|\underline{f}_{\iota}|b_{\iota'},|\overline{f}_{\iota}|b_{\iota'}\}\big], \nonumber\\
  &\bigcap_{\iota\in I}\big[\underline{f}_{i}-\max\{|\underline{f}_{\iota}|b_{i},|\overline{f}_{\iota}|b_{i}\},\overline{f}_{i} -\min\{|\underline{f}_{\iota}|b_{i},|\overline{f}_{\iota}|b_{i}\}\big]\neq\emptyset,\quad\forall i\notin I.
\end{align}
Then, it holds that $\pd D=\breve{\pd}_{\rm s}D(I)$, where
\begin{equation}\label{3.13}
  \breve{\pd}_{\rm s}D(I)=\bigg\{f\in D:\;\phi^{(k)}_{a,b}(f)=V^{(k)}_{a,b},\;\exists a\in\bigcap_{\iota\in I}P_{b}(\iota),\;b\in\pd\overline{B}(1),\;k\in\mathbb{Z}_{++}\text{ s.t. }V^{(k+1)}_{a,b}=V^{(k)}_{a,b}\bigg\},
\end{equation}
and $P_{b}(\iota)$ is given by (\ref{3.9}).
\end{corollary}

\begin{proof}
By Theorem \ref{thm:2}, for any subset $I\subset\mathbb{Z}\cap[1,m]$ it holds that $\pd D=\breve{\pd}_{\rm s}D(\iota)$ for every $\iota\in I$, where $\breve{\pd}_{\rm s}D(\iota)$ is given by (\ref{3.8}). Note that under the conditions in (\ref{3.12}), we can write
\begin{align*}
  \bigcap_{\iota\in I}P_{b}(\iota)&=\Bigg\{a\in\mathbb{R}^{m}:\;
  a_{\iota}=
  \begin{cases}
    \underline{f}_{\iota},&\quad\text{if }b_{\iota}\geq0 \\
    \overline{f}_{\iota},&\quad\text{if }b_{\iota}<0,
  \end{cases}
  ,\;\forall\iota\in I; \\
  &\qquad a_{i}\in\bigcap_{\iota\in I} \big[\underline{f}_{i}-\max\{\underline{f}_{\iota}b_{i},\overline{f}_{\iota}b_{i}\},\overline{f}_{i} -\min\{\underline{f}_{\iota}b_{i},\overline{f}_{\iota}b_{i}\}\big],\;\forall i\notin I\Bigg\}\neq\emptyset.
\end{align*}
Therefore, $\pd D=\bigcap_{\iota\in I}\breve{\pd}_{\rm s}D(\iota)=\breve{\pd}_{\rm s}D(I)$, as in (\ref{3.13}).
\end{proof}

\smallskip

\subsection{An alternative formulation}
\label{sec:4.2}

It is by now clear that calculating the boundary $\pd D$, or equivalently,
the scalarization-implied boundary in (\ref{3.4}), requires solving the scalarized problem in (\ref{3.5}) -- parameterized by $a\in\mathbb{R}^{m}$, $b\in\overline{B}(1)$, and $k\in\mathbb{Z}_{++}$. Given the current form of the scalarization sets in (\ref{3.3}), such a problem (even with dimensionality reduction) is somewhat cumbersome to approach, especially when $f:\mathcal{A}\mapsto\mathbb{R}^{m}$ is considered an objective function from some admissibility set $\mathcal{A}$. The goal of this section is to reformulate the scalarization functions in (\ref{3.2}) in an explicit form so that solution techniques from standard (stochastic) optimization can be applied to facilitate solving (\ref{3.5}).

The following lemma deciphers the structure of the (optimal) scalarization functions.

\begin{lemma}\label{lem:1}
Consider the setting of Theorem \ref{thm:1} and let Assumption \ref{as:1} be in force. For any given $f\in D$, $a\in\mathbb{R}^{m}$, $b\in\pd\overline{B}(1)$, and $k\in\mathbb{Z}_{++}$, either $\phi^{(k)}_{a,b}(f)\in\{0,k\epsilon\}$ or $f-a-b\phi^{(k)}_{a,b}(f)\in\pd K_{b}(\eta)=\{\beta\in\mathbb{R}^{m}:\;(1-\eta)\|\beta\|=\langle\beta,b\rangle\}$, with $\epsilon=r/3$.
\end{lemma}

\begin{proof}
For convenience, let us write $y^{\ast}=\phi^{(k)}_{a,b}(f)$, with $f,a,b,k$ given. Then, by the definition of the corresponding scalarization set in (\ref{3.3}), we have $y^{\ast}\in[0,k\epsilon]$. Thus, we only need to consider the case $y^{\ast}\in(0,k\epsilon)$.

Suppose that $f-a-y^{\ast}b\in\Int K_{b}(\eta)$. Then, by (\ref{2.1}), $(1-\eta)\|f-a-y^{\ast}b\|<\langle f-a-y^{\ast}b,b\rangle$, or equivalently, for some $\varepsilon_{1}>0$,
\begin{equation*}
  (1-\eta)\|f-a-y^{\ast}b\|+\varepsilon_{1}\leq\langle f-a-y^{\ast}b,b\rangle.
\end{equation*}
On the other hand, since $y^{\ast}<k\epsilon$, there exists $\varepsilon_{2}\in(0,\varepsilon_{1}/(2-\eta))$ such that $\hat{y}=y^{\ast}+\varepsilon_{2}\in(0,k\epsilon]$. Then, the fact that $\|b\|=1$, together with the triangle inequality, yields that
\begin{align}\label{4.1}
  (1-\eta)\|f-a-\hat{y}b\|&\leq(1-\eta)\|f-a-y^{\ast}b\|+\varepsilon_{2}(1-\eta)
  \leq(1-\eta)\|f-a-y^{\ast}b\|+\varepsilon_{1}-\varepsilon_{2} \nonumber\\
  &  \leq\langle f-a-y^{\ast}b,b\rangle-\varepsilon_{2}
  =\langle f-a-\hat{y}b,b\rangle,
\end{align}
that is, $\hat{y}\in\mathfrak{S}^{(k)}_{a,b}(f)$, but $\hat{y}>y^{\ast}$, contradicting the maximality of $y^{\ast}$ due to (\ref{3.2}). Therefore, $f-a-y^{\ast}b\in\pd K_{b}(\eta)$.
\end{proof}

The reformulation is presented in Theorem \ref{thm:3} below.

\begin{theorem}\label{thm:3}
Consider the setting of Theorem \ref{thm:1} and let Assumption \ref{as:1} hold. For any $a\in\mathbb{R}^{m}$, $b\in\pd\overline{B}(1)$, and $k\in\mathbb{Z}_{++}$, the value function in (\ref{3.4}) can be equivalently written as
\begin{equation}\label{4.2}
  V^{(k)}_{a,b}=\sup_{f\in D}\phi^{(k)}_{a,b}(f)=\min\Big\{\Big(\sup_{f\in D}H^{\eta}_{a,b}(f)\Big)^{+},k\epsilon\Big\},
\end{equation}
where $(\cdot)^{+}$ denotes the positive part and
\begin{equation}\label{4.3}
  H^{\eta}_{a,b}(f)=\langle f-a,b\rangle-\sqrt{\frac{(1-\eta)^{2}(\|f-a\|^{2}-\langle f-a,b\rangle^{2})}{\eta(2-\eta)}},
\end{equation}
with $\epsilon=r/3>0$ and $\eta\in(0,1]$.
\end{theorem}

\begin{proof}
That $V^{(k)}_{a,b}\le k\epsilon$ is obvious by definition. We need only show that $\phi^{(k)}_{a,b}(\cdot)$ and $H_{a,b}^{\eta}(\cdot)$ have the same maximizer.
In light of Lemma \ref{lem:1}, we know that $y^*=y^*(f)=\phi^{(k)}_{a,b}(f)$ satisfies that $f-a-y^*(f)b\in\pd K_{b}(\eta)$, for  $f\in D$ and given $a,b,k$. That is, $y^*$ solves the following equation:
\begin{equation}\label{4.4}
0\leq  \langle f-a-yb,b\rangle=(1-\eta)\|f-a-yb\|,\quad y\in\mathbb{R}.
\end{equation}
Noting that $\|b\|=1$, by rearranging (\ref{4.4}) we obtain
\begin{equation*}
  \eta(2-\eta)(y^{2}-2\langle f-a,b\rangle y)+\langle f-a,b\rangle^{2}-(1-\eta)^{2}\|f-a\|^{2}=0\quad\text{s.t. }\langle f-a-yb,b\rangle\geq0.
\end{equation*}
We see that $y^*$ must coincide with the (unique) solution $y^*=H_{a,b}^{\eta}(f)\in\mathbb{R}$, thanks to (\ref{4.3}). Combining this with the fact $y^*\in[0,k\epsilon]$, we have
$\phi^{(k)}_{a,b}(f)=\min\big\{(H_{a,b}^{\eta}(f))^{+},k\epsilon\big\}$,
proving (\ref{4.2}).
\end{proof}

\begin{remark}\label{rem:5}
The second term on the right side of (\ref{4.3}) implies that $H^{\eta}_{a,b}(f)$ is generally a nonlinear function of $f\in D$. This feature is a result of the constant $\eta\in(0,1]$, which measures the sharpness of the exterior (circular) cone $K_{\nu}(\eta)$ in Assumption \ref{as:1}, or the ``curvature'' of the (nonconvex) boundary (recall (\ref{3.1})).
\qed
\end{remark}

We now turn our attention to a special application in multi-valued optimization problems in which $f:\mathcal{A}\mapsto\mathbb{R}^{m}$
is a given multi-objective mapping, and $D=\cl\{f(x):\;x\in\mathcal{A}\}$. In such a case, the reformulation (\ref{4.2}) is then to be implemented as
\begin{equation}\label{4.5}
  V^{(k)}_{a,b}=\min\Big\{\Big(\sup_{x\in\mathcal{A}}H^{\eta}_{a,b}(f(x))\Big)^{+},k\epsilon\Big\},\quad k\in\mathbb{Z}_{++},
\end{equation}
where $H^{\eta}_{a,b}$ is given by (\ref{4.3}).

Of course, following the discussion in Section \ref{sec:3}, in the case where $\epsilon\to\infty$ is allowed (e.g., if $D$ is bean-shaped or even a convex domain), then there is only one scalarized problem to solve in (\ref{3.4}) for each parameter $a$ and $b$, and consequently (\ref{4.5}) also reduces to a single problem: $V^{(1)}_{a,b}=\big(\sup_{x\in\mathcal{A}}H^{\eta}_{a,b}(f(x))\big)^{+}$. Further, when $D$ is convex, then we can simply take $\eta=1$, and the second term in (\ref{4.3}) disappears, making $H^{\eta}_{a,b}(f)=\langle f-a,b\rangle$ a linear function; hence, (\ref{4.5}) becomes a standard linear optimization problem: $V^{(1)}_{a,b}=\big(\sup_{x\in\mathcal{A}}\langle f-a,b\rangle\big)^{+}$.


Based on (\ref{4.5}), we proceed to consider the following (single-criterion) optimization problem:
\begin{equation}\label{4.6}
  F_{a,b}:=\sup_{x\in\mathcal{A}}H^{\eta}_{a,b}(f(x)),
\end{equation}
with given parameters $a,b$ from the scalarization and the function in (\ref{4.3}). To simplify things, let us further denote
\begin{equation}\label{4.7}
  \theta=\frac{1-\eta}{\sqrt{\eta(2-\eta)}}\geq0,
\end{equation}
which allows to rewrite the function as
\begin{equation*}
  H^{\eta}_{a,b}(f(x))=\langle f(x)-a, b\rangle-\theta \sqrt{\|f(x)-a\|^2-\langle f(x)-a, b\rangle^2}=\langle f(x)-a, b\rangle-\theta\|\mbox{Proj}_{b^\perp}\{f(x)-a\}\|,
\end{equation*}
where $b^\perp\subset\mathbb{R}^m$ stands for the $(m-1)$-dimensional orthogonal subspace of $b\in\mathbb{R}^{m}$. Now, by letting $\{b^{\perp,(j)}:j\in\mathbb{Z}\cap[1,m-1]\}$ be an orthonormal basis of $b^\perp$,\footnote{Equivalently, $\{b^{\perp,(j)}:j\in\mathbb{Z}\cap[1,m-1]\}\cup\{b\}\subset\overline{B}(1)$ forms an orthonormal basis of $\mathbb{R}^{m}$.} we can recast (\ref{4.6}) into the form
\begin{equation}\label{4.8}
  F_{a,b}=\sup_{x\in\mathcal{A}}\Bigg\{\langle f(x)-a,b\rangle-\theta\sqrt{\sum^{m-1}_{j=1}\langle f(x)-a,b^{\perp,(j)}\rangle^{2}}\Bigg\}.
\end{equation}

The optimization problem (\ref{4.8}) can be solved in two stages. In the first stage, we fix an arbitrary $l=(l_{1},\dots,l_{m-1})^{\intercal}$ (vector of control) and solve the following (possibly stochastic) constrained problem with exactly $m-1$ linear constraints:
\begin{equation*}
  G_{a,b}(l):=\sup_{x\in\mathcal{A}}\langle f(x),b\rangle\quad\text{s.t.}\quad\langle f(x)-a,b^{\perp,(j)}\rangle=l_{j}\in\mathbb{R},\;j\in\mathbb{Z}\cap[1,m-1];
\end{equation*}
we make the convention that $G_{a,b}(l)=-\infty$ if any constraint fails to hold by any $x\in\mathcal{A}$. By defining the Lagrangian with Lagrange multipliers $\lambda_{l_{j}}\in\mathbb{R}$, $j\leq m-1$, we can equivalently consider the unconstrained problem
\begin{equation}\label{4.9}
  G_{a,b}(l):=\sup_{x\in\mathcal{A}}\Bigg\{\langle f(x),b\rangle+\sum^{m-1}_{j=1}\lambda_{l_{j}}(\langle f(x),b^{\perp,(j)}\rangle-\langle a,b^{\perp,(j)}\rangle-l_{j})\Bigg\}.
\end{equation}
The solution to (\ref{4.9}) is understood to be $x^{\ast}\equiv x^{\ast}(\lambda^{\ast}_{l})$, where $\lambda^{\ast}_{l}=(\lambda^{\ast}_{l_{1}},\dots,\lambda^{\ast}_{l_{m-1}})^{\intercal}$ is such that $\langle f(x^{\ast}(\lambda^{\ast}_{l})),b^{\perp,(j)}\rangle=\langle a,b^{\perp,(j)}\rangle+l_{j}$.

The second stage then consists in solving an unconstrained problem with respect to the vector $l$, i.e.,
\begin{equation}\label{4.10}
  F_{a,b}=\sup_{l\in\mathbb{R}^{m-1}}\big\{G_{a,b}(l)-\langle a,b\rangle-\theta\|l\|\big\},
\end{equation}
where $\theta$ satisfies (\ref{4.7}). Only the problem in (\ref{4.9}) may involve stochastic optimization, while (\ref{4.10}) is entirely deterministic.

We also note that the optimal control $x^{\ast}$ obtained from solving (\ref{4.10}) is generally \textsl{not} exhaustive for the original problem $V^{(k)}_{a,b}=\sup_{x\in\mathcal{A}}\phi^{(k)}_{a,b}(f(x))$, for $k\in\mathbb{Z}_{++}$ given, as suggested by (\ref{3.4}), simply because the function $\min\{(\cdot)^{+},k\epsilon\}$, despite being monotone, is not strictly so. Nonetheless, the value functions can still be completely recovered as
\begin{equation*}
  V^{(k)}_{a,b}=\min\big\{F^{+}_{a,b},k\epsilon\big\},\quad k\in\mathbb{Z}_{++},
\end{equation*}
where $F_{a,b}$ is the outcome of (\ref{4.10}), with $a,b$ given.

\smallskip

\subsection{A numerical example}\label{sec:4.3}

In this section, we design a 2D deterministic optimization example as a test problem for the proposed characterization and the solution techniques discussed above.

Consider a 2D admissibility set $\mathcal{A}=\{(x_{1},x_{2})^{\intercal}\in\mathbb{R}^{2}_{+}:\;x_{1}+x_{2}\leq1\}$, along with the (objective) function
\begin{equation}\label{4.1.1}
  f(x)=\bigg(\sqrt{x^{2}_{1}+2x^{2}_{2}},\;\cos(2x_{1}+x^{2}_{2})-e^{-x^{2}_{2}}+\frac{\sin(3x_{1}x_{2})}{3}\bigg)^{\intercal},\quad x\in\mathcal{A},
\end{equation}
whose second component is a highly nonconvex function, and so the desired boundary $\pd D=\pd\cl\{f(x):\;x\in\mathcal{A}\}\subset\mathbb{R}^{2}$ will likely have nonconvex fronts.

Assume that $D$ satisfies condition (\ref{3.1}) with arbitrary $r>0$ and $\eta=1-\cos(\pi/8)$.\footnote{In practice, assumptions alike can always be verified numerically by reducing the values and observing whether significant changes occur. Nonetheless, we omit these details here as the current focus is not on numerics.} By consulting Theorem \ref{thm:1} and Theorem \ref{thm:3}, the boundary can be written as
\begin{align}\label{4.1.2}
  \pd D&=\cl\Big\{f(x):\;x\in\underset{x\in\mathbb{R}^{2}_{+},\;x_{1}+x_{2}\leq1}{\arg\max}H^{\eta}_{a,b}(f(x)),\;\exists a\in\mathbb{R}^{2},\;b\in\pd\overline{B}(1)\text{ s.t. }H^{\eta}_{a,b}(f(x))\geq0\Big\} \nonumber\\
  &=\cl\Big\{f(x):\;x\in\underset{x\in\mathbb{R}^{2}_{+},\;x_{1}+x_{2}\leq1}{\arg\max}H^{\eta}_{a,b}(f(x)),\;\exists a\in P_{b}(\iota),\;b\in\pd\overline{B}(1)\text{ s.t. }H^{\eta}_{a,b}(f(x))\geq0\Big\},
\end{align}
where the second equality follows from the (clear) boundedness of $D$ and Theorem \ref{thm:2}, with $P_{b}(\iota)$ being the parametric range in (\ref{3.9}) and the function $H^{\eta}_{a,b}$ is given by (\ref{4.3}).

\begin{figure}[H]
  \centering
  \includegraphics[scale=0.4]{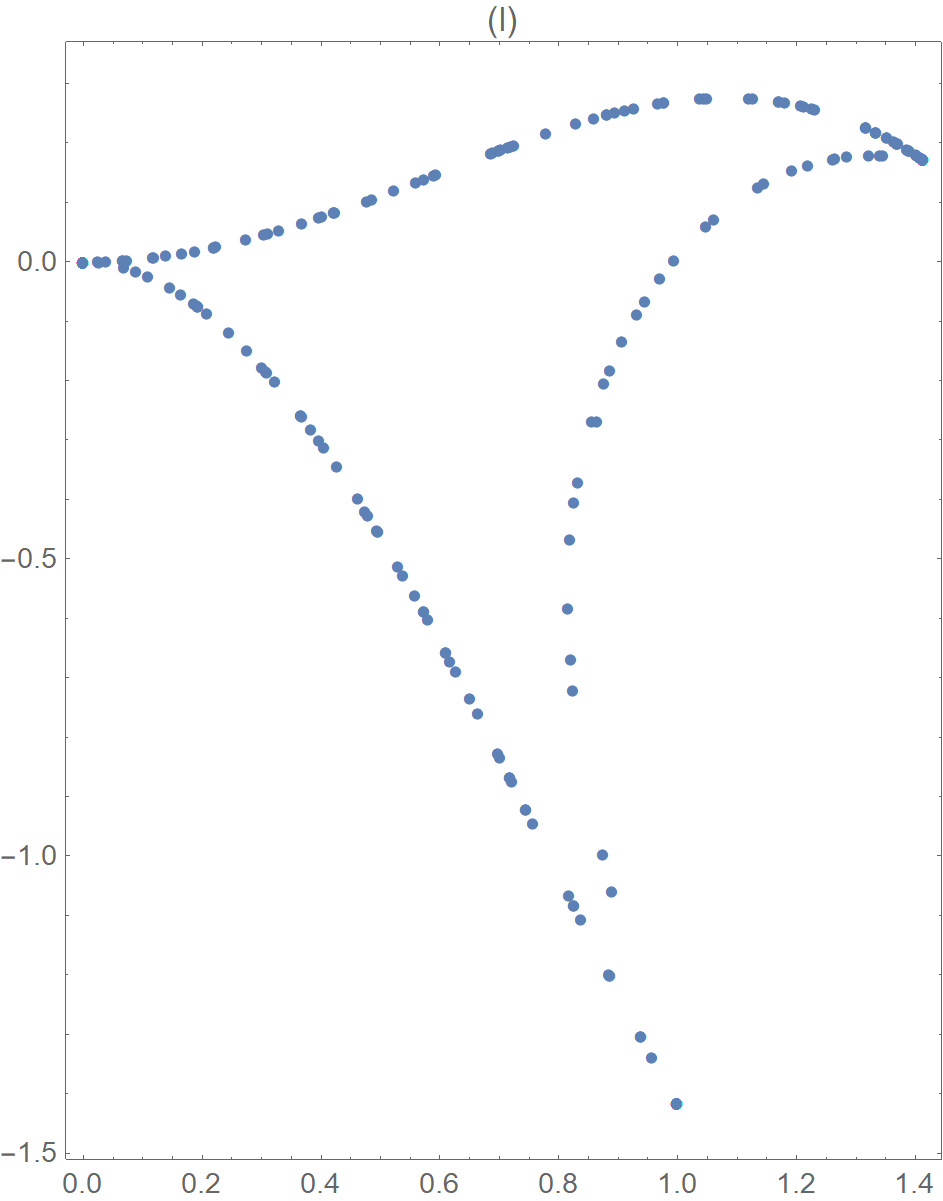} \quad
  \includegraphics[scale=0.4]{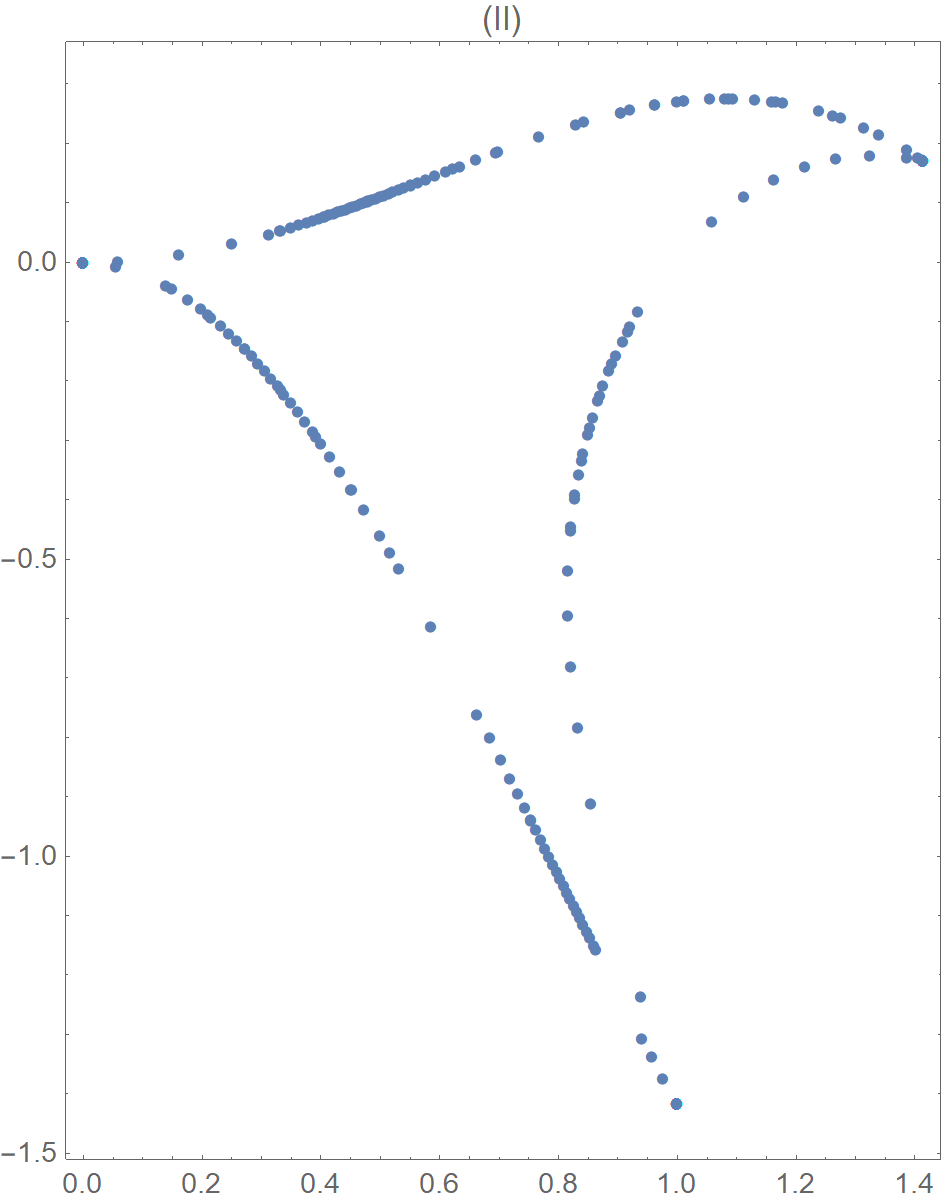} \\
  \caption{Approximations of scalarization-implied boundary}
  \label{fig:4}
\end{figure}

Due to the relative simplicity of the first component of $f$ (in (\ref{4.1.1})), we take $\iota=2$. It is then easy to verify that $P_{b}(2)\subset[-7/3,\sqrt{2}+7/3]\times\{-7/3,4/3\}$, for any $b\in\pd\overline{B}(0)$. To generate an approximation of (\ref{4.1.2}), we use 100 evenly spaced unit vectors for the orient, $b_{i}=(\cos(2\pi i/100),\sin(2\pi i/100))^{\intercal}$, for $i\in\mathbb{Z}\cap[1,100]$; for the base $a$, depending on whether $b_{2}\geq0$ or $b_{2}<0$, we set $a_{2}=-7/3$ or $a_{2}=4/3$, respectively, and evaluate nine uniformly spaced values of $a_{1}$ within the interval $[-7/3,\sqrt{2}+7/3]$. Each optimization problem involved in (\ref{4.1.2}) is then solved numerically by following the procedures outlined before (see (\ref{4.9}) and (\ref{4.10})), with $m=2$. The approximated scalarization-implied boundary is shown in the left panel of Figure \ref{fig:4}. For comparison, we repeat the approximation using four randomly chosen points $\{(0,0)^{\intercal},(1,1)^{\intercal},(1,-1)^{\intercal},(1/2,0)^{\intercal}\}\ni a$, while keeping other things equal; results are shown in the right panel.

It is seen that the approximation based on the reduced parametric range $P_{b}(2)$ tends to generate more evenly distributed boundary points, while both approximations effectively capture the shape of the boundary $\pd D$.

\medskip

\section{Extension to infinite-dimensional spaces}\label{sec:5}

The results in the previous sections can be extended to a topological vector space $\mathcal{X}$ in lieu of $\mathbb{R}^{m}$, with $D\in\mathrm{Cl}(\mathcal{X})$. While such a topological vector space can be arbitrarily general in theory, we require meaningful connections to the finite-dimensional setting to facilitate imposing a suitable topological structure and for the subsequent analysis to remain well-founded.

We give some motivation for considering such dimensional infiniteness here in the context of games. Consider a generic game with $N\geq2$ players indexed by $i=n\in\mathbb{Z}\cap[1,N]$. Let $\alpha_{i}\in A_{i}$ denote the strategy of the $i$th player, who aims to maximize some objective function $J_{i}:\prod^{N}_{i=1}A_{i}\mapsto\mathbb{R}$, and suppose that $(\alpha^{\star}_{i};\alpha^{\star}_{\neg i})$ denotes the tuple of the $i$th player's optimal strategy given the (optimal) strategies of the other $N-1$ players. The equilibrium is understood in the sense of Nash as some optimal strategy profile $\alpha^{\star}\in\prod^{N}_{i=1}A_{i}$, and the corresponding game values are given by $(J_{1}(\alpha^{\star}_{1}),\dots,J_{N}(\alpha^{\star}_{N}))\in\mathbb{R}^{N}$; see, e.g.,
\cite{FRZ22} and 
\cite[\S3]{IZ24a}.

In many applications such as finance and economics, the players can be considered indistinguishable in that they have the same dynamics and objective functions, except for randomness in their individual states. Thus, instead of emphasizing the index $n$ of each player, what matters is the empirical distribution of states $\mathbf{m}_{N}=(1/N)\sum^{N}_{n=1}\delta_{i_{n}}$, where $i$ becomes the state of the $n$th player and $\delta$ denotes the dirac measure. In very large populations, as $N\to\infty$, the empirical distribution of states converges to some deterministic measure $\mathbf{m}$, which allows to directly track the density of players -- instead of tracking finitely many individuals. As a result, the game values are now in some function space with a domain $\bar{\mathcal{I}}$ for the continuous-valued state $i$, on which integrability and continuity structures can be conveniently applied; see, e.g., 
\cite[\S4]{IZ24a}. More detailed discussion can be found in Section \ref{sec:6.1}.

With the above consideration, we now let $\bar{\mathcal{I}}\in\mathrm{Conv}(\mathbb{R}^{m})$ be a convex domain, assumed to be of full dimensionality, i.e., $\mathrm{Leb}_{\mathbb{R}^{m}}(\bar{\mathcal{I}})>0$. We shall focus on two cases: the Lebesgue space of square-integrable functions and the space of bounded continuous functions, denoted respectively as $\mathcal{X}=\mathbb{L}^{2}(\bar{\mathcal{I}};\mathbb{R})$ and $\mathcal{X}=\mathcal{C}_{\rm b}(\bar{\mathcal{I}};\mathbb{R})$. We note that $\mathbb{L}^{2}(\bar{\mathcal{I}};\mathbb{R})$ is a Hilbert space equipped with the usual inner product $\langle\cdot,\cdot\rangle_{\bar{\mathcal{I}}}$ over the domain $\bar{\mathcal{I}}$ and the $\mathbb{L}^{2}$-norm $\|\cdot\|_{2}$. On the other hand, $\mathcal{C}_{\rm b}(\bar{\mathcal{I}};\mathbb{R})$ is a Banach space equipped with the uniform norm $\|\cdot\|_{\infty}$, and its \textsl{dual space} is the space $\mathcal{M}(\bar{\mathcal{I}};\mathbb{R})$ of (signed) Radon measures of bounded variation restricted to $\bar{\mathcal{I}}$, equipped with the total variation norm $\|\cdot\|_{\rm TV}$. In this case, we shall reuse $\langle\cdot,\cdot\rangle_{\bar{\mathcal{I}}}$ to denote the \textsl{duality pairing} between $\beta\in\mathcal{C}_{\rm b}(\bar{\mathcal{I}};\mathbb{R})$ and $\nu\in\mathcal{M}(\bar{\mathcal{I}};\mathbb{R})$ over the domain $\bar{\mathcal{I}}$, understood as the Stieltjes integral $\langle\beta,\nu\rangle_{\bar{\mathcal{I}}}=\int_{\bar{\mathcal{I}}}\beta_{i}\nu(\dd i)$. Furthermore, we shall emphasize the function space in the closure, interior, and boundary operators by writing the subscripts ``$\mathbb{L}^{2}$'' and ``$\mathcal{C}_{\rm b}$'' accordingly, as $\cl_{\mathbb{L}^{2}}$, $\cl_{\mathcal{C}_{\rm b}}$, etc.

The notion of convex cones and pointedness readily extends to these two function spaces. For $\mathcal{X}=\mathbb{L}^{2}(\bar{\mathcal{I}};\mathbb{R})$ and $ \mathcal{C}_{\rm b}(\bar{\mathcal{I}};\mathbb{R})$, respectively, we define
a generic circular cone by
\begin{equation}\label{5.1}
  {K}_{\mathcal{X},\nu}(\eta):=\{\beta\in\mathcal{X}:\;  (1-\eta)\|\beta\|_{\mathcal{X}}\leq\langle\beta,\nu\rangle_{\bar{\mathcal{I}}}\},
\end{equation}
where $\eta\in(0,1]$, and $\nu\in \mathcal{X}^*$ (dual space) with $\|\nu\|_{\mathcal{X}^*}=1$. Clearly, for the spherical cone in $\mathcal{X}=\mathbb{L}^{2}(\bar{\mathcal{I}};\mathbb{R})$, we can take the closed ball  $\overline{B}(r)=\{\beta\in\mathbb{L}^{2}(\bar{\mathcal{I}};\mathbb{R}):\;\|\beta\|_{2}\leq r\}$, and by the Rietz representation, in (\ref{5.1}) $\langle\cdot, \cdot\rangle_{\bar{\mathcal{I}}}$ is simply understood as the $\mathbb{L}^{2}$-inner product, and $\nu\in\mathbb{L}^{2}(\bar{\mathcal{I}};\mathbb{R})$ with $\|\nu\|_{2}=1$. However, since
 $\mathcal{X}=\mathcal{C}_{\rm b}(\bar{\mathcal{I}};\mathbb{R})$ is nonreflexive with $\mathcal{X}^*=\mathcal{M}(\bar{\mathcal{I}};\mathbb{R})$, we shall consider the closed balls
 $\overline{B}_{\mathcal{C}_{\rm b}}(r)=\{\beta\in\mathcal{C}_{\rm b}(\bar{\mathcal{I}};\mathbb{R}):\;\|\beta\|_{\infty}\leq r\}$ and $\overline{B}_{\mathcal{M}}(r)=\{\nu\in\mathcal{M}(\bar{\mathcal{I}};\mathbb{R}):\;\|\nu\|_{\rm TV}\leq r\}$, respectively.
 In this case, the   bracket $\langle\cdot, \cdot\rangle_{\bar{\mathcal{I}}}$ in (\ref{5.1}) is understood as the duality pairing.

Before proceeding further, we would like to point out that, although it is well-known that the positive cone $\mathbb{L}^{2}(\bar{\mathcal{I}};\mathbb{R}_{+})$ has an empty interior,
it is clear that $\Int_{\mathbb{L}^{2}}K_{\mathbb{L}^{2}, \nu}(\eta)=\{\beta\in\mathbb{L}^{2}(\bar{\mathcal{I}};\mathbb{R}):\; (1-\eta)\|\beta\|_{2}<\langle\beta,\nu\rangle_{\bar{\mathcal{I}}}\}$ and $\Int_{\mathcal{C}_{\rm b}}K_{ \mathcal{C}_{\rm b},\nu}(\eta)=\{\beta\in\mathcal{C}_{\rm b}(\bar{\mathcal{I}};\mathbb{R}):\; (1-\eta)\|\beta\|_{\infty}<\langle\beta,\nu\rangle_{\bar{\mathcal{I}}}\}$, respectively, are both nonempty. Thus, the corresponding spherical cones $ {K}^r_{\mathbb{L}^{2}, \nu}(\eta)$ and ${K}^r_{\mathcal{C}_{\rm b},\nu}(\eta)$ (\text{resp.} $\breve{K}^r_{\mathbb{L}^{2},\nu}(\eta)$ and $\breve{K}^r_{\mathcal{C}_{\rm b},\nu}(\eta)$), defined by intersecting the respective circular cones (\text{resp.} their interiors) with the closed balls $\overline{B}_{\mathbb{L}^{2}}(r)$ and $\overline{B}_{\mathcal{C}_{\rm b}}(r)$, $r\in(0,\infty]$, are all meaningful.

For $\mathcal{X}=\mathbb{L}^{2}(\bar{\mathcal{I}};\mathbb{R})$ or $\mathcal{C}_{\rm b}(\bar{\mathcal{I}};\mathbb{R})$, we can define multi-criteria optimization problems associated with the partial order $\succeq_{r}\equiv\succeq_{K^{r}_{\mathcal{X},\nu}(\eta)}$, similar to (\ref{2.2}). More precisely, we consider the problem
\begin{equation}\label{5.3}
  \sup_{x\in\mathcal{A}}f(x)\equiv\sup_{x\in\mathcal{A}}(f_{i}(x):\;i\in\bar{\mathcal{I}}),
\end{equation}
where $\mathcal{A}$ is the same admissibility set,  $f:\mathcal{A}\mapsto\mathcal{X}$ is the objective function, and $D=\cl_{\mathcal{X}}\{f(x):\; x\in\mathcal{A}\}$. The notion of maximality when $\mathcal{X}=\mathbb{L}^{2}(\bar{\mathcal{I}};\mathbb{R})$ or $\mathcal{X}=\mathcal{C}_{\rm b}(\bar{\mathcal{I}};\mathbb{R})$ is understood as follows.

\begin{definition}
\label{def:3}
For $\mathcal{X}=\mathbb{L}^{2}(\bar{\mathcal{I}};\mathbb{R})$ or $\mathcal{C}_{\rm b}(\bar{\mathcal{I}};\mathbb{R})$, let ${K}^r_{\mathcal{X}, \nu}(\eta)$ (\text{resp.} $\breve{K}^r_{\mathcal{X}, \nu}(\eta)$)
be the spherical cone (\text{resp.} its partially open counterpart) based on (\ref{5.1}) and $D\in\mathrm{Cl}(\mathcal{X})$. A control $x^{\ast}\in\mathcal{A}$ is said to be a (weakly) ${K}^r_{\mathcal{X},\nu}(\eta)$-maximal solution to the problem (\ref{5.3}) if $(f(x^{\ast})+\breve{K}^r_{\mathcal{X},\nu}(\eta))\cap D=\emptyset$.
\qed
\end{definition}

We now try to extend the boundary characterization results in Section \ref{sec:3} and Section \ref{sec:4.2} into the infinite-dimensional setting. Let us first consider the case $\mathcal{X}=\mathbb{L}^{2}(\bar{\mathcal{I}};\mathbb{R})$, where $\bar{\mathcal{I}}\in\mathrm{Conv}(\mathbb{R}^{m})$ satisfies $\mathrm{Leb}_{\mathbb{R}^{m}}(\bar{\mathcal{I}})>0$. By reflexivity, the following regularity assumption is natural following Assumption \ref{as:1}.
\begin{assumption}
\label{as:2}
Let $D\in\mathrm{Cl}(\mathbb{L}^{2}(\bar{\mathcal{I}};\mathbb{R}))$. There exist constants $r>0$ and $\eta\in(0,1]$ such that for every $f\in\pd D$,
\begin{equation}\label{5.4}
  (f+\breve{K}^r_{\mathbb{L}^{2},\nu}(\eta))\cap D=\emptyset,
\end{equation}
for some $\nu\equiv\nu_{f}\in\pd_{\mathbb{L}^{2}}\overline{B}_{\mathbb{L}^{2}}(1)$, with $\pd_{\mathbb{L}^{2}}\overline{B}_{\mathbb{L}^{2}}(1)=\{\beta\in\mathbb{R}^{m}:\;\|\beta\|_{2}=1\}$ and $\breve{K}^r_{\mathbb{L}^{2},\nu}(\eta)$ based on (\ref{5.1}).
\qed
\end{assumption}

Analogous to (\ref{3.2}), to characterize the boundary $\pd_{\mathbb{L}^{2}}D$ we introduce the
scalarization sets
\begin{equation}\label{5.6}
  \mathfrak{S}^{(k)}_{a,b}(f)=\{y\in[0,k\epsilon]:\;f-a-yb\in{K}^{\epsilon}_{\mathbb{L}^{2},b}(\eta)\},\quad f\in D,\;k\in\mathbb{Z}_{++},\;\epsilon=\frac{r}{3},
\end{equation}
where $a\in\mathbb{L}^{2}(\bar{\mathcal{I}};\mathbb{R})$ and $b\in\pd_{\mathbb{L}^{2}}\overline{B}_{\mathbb{L}^{2}}(1)$,
and then define the scalarization functions
\begin{equation}\label{5.5}
  \phi^{(k)}_{a,b}(f):=\sup_{y\in\mathfrak{S}^{(k)}_{a,b}(f)}y,\quad a\in\mathbb{L}^{2}(\bar{\mathcal{I}};\mathbb{R}),~b\in\pd_{\mathbb{L}^{2}}\overline{B}_{\mathbb{L}^{2}}(1),~k\in\mathbb{Z}_{++}.
\end{equation}
Then, the scalarization-implied boundary is defined as
\begin{equation}\label{5.7}
  \pd_{\mathbb{L}^{2},\rm s}D:=\big\{f\in D:\;\phi^{(k)}_{a,b}(f)=V^{(k)}_{a,b},\;\exists a\in\mathbb{L}^{2}(\bar{\mathcal{I}};\mathbb{R}),\;b\in\pd_{\mathbb{L}^{2}}\overline{B}_{\mathbb{L}^{2}}(1),\;k\in\mathbb{Z}_{++}\text{ s.t. }V^{(k+1)}_{a,b}=V^{(k)}_{a,b}\big\},
\end{equation}
where, for $k\in\mathbb{Z}_{++}$,
\begin{equation}\label{5.8}
  V^{(k)}_{a,b}=\sup_{f\in D}\phi^{(k)}_{a,b}(f).
\end{equation}

\smallskip

For the case $\mathcal{X}=\mathcal{C}_{\rm b}(\bar{\mathcal{I}};\mathbb{R})$, the orient $\nu$ of the circular cone in (\ref{5.1}) is understood in the dual space $\mathcal{X}^*=\mathcal{M}(\bar{\mathcal{I}};\mathbb{R})$, and $\langle\cdot,\cdot\rangle$ is the duality pairing. The regularity assumption then becomes the following.

\begin{assumption}\label{as:3}
Let $D\in\mathrm{Cl}(\mathcal{C}_{\rm b}(\bar{\mathcal{I}};\mathbb{R}))$. There exist constants $r>0$ and $\eta\in(0,1]$ such that for every $f\in\pd D$,
\begin{equation}\label{5.9}
  (f+\breve{K}^r_{\mathcal{C}_{\rm b},\nu}(\eta))\cap D=\emptyset,
\end{equation}
for some $\nu\in\pd_{\mathcal{M}}\overline{B}_{\mathcal{M}}(1)$, with $\pd_{\mathcal{C}_{\rm b}}\overline{B}_{\mathcal{C}_{\rm b}}(1)=\{\beta\in\mathbb{R}^{m}:\;\|\beta\|_{\infty}=1\}$ and $\breve{K}^r_{\mathcal{C}_{\rm b},\nu}(\eta)$ also based on (\ref{5.1}).
\end{assumption}

We stress that since $\mathcal{C}_{\rm b}(\bar{\mathcal{I}};\mathbb{R})$ is non-reflexive, there are some subtle points that need to be addressed carefully. For example, if we follow the idea of (\ref{3.3}) or (\ref{5.6}) to define the scalarization sets $\mathfrak{S}^{(k)}_{a,b}(f)$ for generic functional parameters $a\in\mathcal{C}_{\rm b}(\bar{\mathcal{I}};\mathbb{R})$ and $b\in\pd_{\mathcal{C}_{\rm b}}\overline{B}_{\mathcal{C}_{\rm b}}(1)$, then an immediate difficulty is that $b$ cannot be used as a shifting direction in $\mathcal{C}_{\rm b}(\bar{\mathcal{I}};\mathbb{R})$ because the orient of the circular cone is supposed to take values in the dual space $\mathcal{M}(\bar{\mathcal{I}};\mathbb{R})$.

To overcome this difficulty, we link each $b\in\pd_{\mathcal{C}_{\rm b}}\overline{B}_{\mathcal{C}_{\rm b}}(1)$ to a
Radon measure $\mathfrak{b}\in\mathcal{M}(\bar{\mathcal{I}};\mathbb{R})$ (orient), whose \textsl{Radon--Nikodym derivative} with respect to the Lebesgue measure is $cb\in\pd_{\mathcal{C}_{\rm b}}\overline{B}_{\mathcal{C}_{\rm b}}(c)$, with some positive scaling factor $c\equiv c_{\bar{\mathcal{I}}}$ depending on $\bar{\mathcal{I}}$.  More precisely, $\mathfrak{b}(A)=c\int_A b_i\dd i$ for any measurable subset $A\subset\bar{\mathcal{I}}$. Since $\|b\|_{\infty}=1$, it is clear that $\int_{\bar{\mathcal{I}}}b^{2}_{i}\dd i\leq\int_{\bar{\mathcal{I}}}|b_{i}|\dd i$, and hence
\begin{equation}\label{5.11}
  \|\mathfrak{b}\|_{\rm TV}=\frac{\int_{\bar{\mathcal{I}}}|b_{i}|\dd i}{\int_{\bar{\mathcal{I}}}b^{2}_{i}\dd i}\geq1,
\end{equation}
which implies that $c=1\big/\int_{\bar{\mathcal{I}}}b^{2}_{i}\dd i>0$. Then, we can consider the revised scalarization sets
\begin{equation}\label{5.10}
  \mathfrak{S}^{(k)}_{a,b}(f)=\{y\in[0,k\epsilon]:\;f-a-yb\in {K}^{\epsilon}_{\mathcal{C}_{\rm b},\mathfrak{b}}(\eta)\},\quad f\in D,\;k\in\mathbb{Z}_{++},\;\epsilon=\frac{r}{3}.
\end{equation}

Next, we proceed to define the corresponding nonlinear scalarization functions, parameterized by $a$ and $b$, as well as the associated ``value functions'':
\begin{equation}\label{5.12}
  \phi^{(k)}_{a,b}(f):=\sup_{y\in\mathfrak{S}^{(k)}_{a,b}(f)}y,\qquad
  V^{(k)}_{a,b}=\sup_{f\in D}\phi^{(k)}_{a,b}(f),\quad k\in\mathbb{Z}_{++}.
\end{equation}
From these the scalarization-implied boundary can be defined accordingly,
\begin{equation}\label{5.13}
  \pd_{\mathcal{C}_{\rm b},\rm s}D:=\big\{f\in D:\;\phi^{(k)}_{a,b}(f)=V^{(k)}_{a,b},\;\exists a\in\mathcal{C}_{\rm b}(\bar{\mathcal{I}};\mathbb{R}),\;b\in\pd_{\mathcal{C}_{\rm b}}\overline{B}_{\mathcal{C}_{\rm b}}(1),\;k\in\mathbb{Z}_{++}\text{ s.t. }V^{(k+1)}_{a,b}=V^{(k)}_{a,b}\big\}.
\end{equation}

The following result is the generalization of Theorem \ref{thm:1} in infinite dimensions, governing the equivalence between the original and scalarization-implied boundaries.

\begin{theorem}\label{thm:4}
Under Assumption \ref{as:2}, it holds that $\pd_{\mathbb{L}^{2}}D=\pd_{\mathbb{L}^{2},\rm s}D$. Similarly, under Assumption \ref{as:3}, it holds that $\pd_{\mathcal{C}_{\rm b}}D=\pd_{\mathcal{C}_{\rm b},\rm s}D$.
\end{theorem}

\begin{proof}
We only argue for the second case, i.e., $\mathcal{X}=\mathcal{C}_{\rm b}(\bar{\mathcal{I}};\mathbb{R})$. The proof for the case $\mathcal{X}=\mathbb{L}^2(\bar{\mathcal{I}};\mathbb{R})$ follows the same idea and is closer to the finite-dimensional case, hence inherently easier due to the reflexivity.

\textbf{Step 1.} First, consider an arbitrary $f^{\circ}\in\pd D$, and for a (nonzero) Radon measure $\mu$, let $H_{\mu}(f^{\circ}):=\{f\in\mathcal{C}_{\rm b}(\bar{\mathcal{I}};\mathbb{R}):\;\mu(f)=\mu(f^{\circ})\}$ be a supporting half-space of $D$ that contains $f^{\circ}$. According to Assumption \ref{as:3}, there exists a (not necessarily unique) Radon measure $\nu_{f^{\circ}}$ satisfying that $\|\nu_{f^{\circ}}\|_{\rm TV}=1$, that $\langle f,\nu_{f^{\circ}}\rangle_{\bar{\mathcal{I}}}\geq0$ for all $f\in H_{\mu}(f^{\circ})$, and that $(f^{\circ}+\breve{K}^{3\epsilon}_{\mathcal{C}_{\rm b},\nu_{f^{\circ}}}(\eta))\cap D=\emptyset$.

Moreover, since $\bar{\mathcal{I}}$ has positive Lebesgue measure (in $\mathbb{R}^{m}$), the Lebesgue decomposition theorem implies that $\mu$ can be chosen to be absolutely continuous with a continuous Radon--Nikodym derivative (say $\gamma\in\mathcal{C}_{\rm b}(\bar{\mathcal{I}};\mathbb{R})$). If so, we can also require the above measure $\nu_{f^{\circ}}$ to be absolutely continuous with a continuous Radon--Nikodym derivative. Indeed, in such a case, $H_{\mu}(f^{\circ})\equiv H_{\gamma}(f^{\circ})$ is restricted to the subspace $(\mathcal{C}_{\rm b}\cap\mathbb{L}^{2})(\bar{\mathcal{I}};\mathbb{R})$ equipped with the norm $\|\cdot\|_{\infty}+\|\cdot\|_{2}$, and $\langle\cdot,\cdot\rangle_{\bar{\mathcal{I}}}$ can be understood as the $\mathbb{L}^{2}$-inner product, based on the Radon--Nikodym derivative. Since this is obviously a closed (with respect to the norm topology) subspace of the (Hilbert) space $\mathbb{L}^{2}(\bar{\mathcal{I}};\mathbb{R})$, hence reflexive, the Radon--Nikodym derivative of $\nu_{f^{\circ}}$ is also a square-integrable bounded continuous function.

Thus, we choose $a=f^{\circ}$ and $\mathfrak{b}=\nu_{f^{\circ}}$, with the understanding that $\nu_{f^{\circ}}(A)=c\int_{A}b_{i}\dd i$ and $\|b\|_{\infty}=1$ for any measurable subset $A\subset\bar{\mathcal{I}}$, and $cb\in\pd_{\mathcal{C}_{\rm b}}\overline{B}_{\mathcal{C}_{\rm b}}(c)$. Then, from (\ref{5.10}) we have that
\begin{equation*}
  \mathfrak{S}^{(1)}_{f^{\circ},b}(f^{\circ})=\{y\in[0,\epsilon]:\;-yb\in {K}^\epsilon_{\mathcal{C}_{\rm b},\mathfrak{b}}(\eta)\};
\end{equation*}
however, since from (\ref{5.1}),
\begin{equation}\label{5.15}
 {K}_{\mathcal{C}_{\rm b}, \mathfrak{b}}(\eta)=\bigg\{\beta\in\mathcal{C}_{\rm b}(\bar{\mathcal{I}};\mathbb{R}):\;(1-\eta)\|\beta\|_{\infty}\leq c\int_{\bar{\mathcal{I}}}\beta_{i}b_{i}\dd i\bigg\},
\end{equation}
then the condition $-yb\in{K}_{\mathcal{C}_{\rm b}, \mathfrak{b}}(\eta)$ amounts to $0\le c\int_{\bar{\mathcal{I}}}(-yb_{i})b_{i}\dd i =-y\le 0$, which is only possible when $y=0$. In other words, we must have $\mathfrak{S}^{(1)}_{f^{\circ},b}(f^{\circ})=\{0\}$, and hence $V^{(1)}_{f^{\circ},b}\geq\phi^{(1)}_{f^{\circ},b}(f^{\circ})=0$.

We now claim that $\mathfrak{S}^{(1)}_{f^{\circ},b}(f)=\{0\}$, for all $f\in D$, which will lead to $V^{(1)}_{f^{\circ},b}=\phi^{(1)}_{f^{\circ},b}(f^{\circ})=0$. Indeed, since $f^{\circ}\in\pd D$, condition (\ref{5.9}) implies that $f-f^{\circ}\notin \breve{K}^{3\epsilon}_{\mathcal{C}_{\rm b},\mathfrak{b}}(\eta)$ for any $f\in D$, and from (\ref{5.10}), $f-f^{\circ}\in yb+ {K}^\epsilon_{\mathcal{C}_{\rm b},\mathfrak{b}}(\eta)$ for any $f\in D$, $y\in\mathfrak{S}^{(1)}_{f^{\circ},b}(f)$, and some $\beta\in{K}^\epsilon_{\mathcal{C}_{\rm b},\mathfrak{b}}(\eta)$. Note that if 
$\beta\in{K}^\epsilon_{\mathcal{C}_{\rm b},\mathfrak{b}}(\eta)$,  $y\in[0,2\epsilon]$, and $\|b\|_{\infty}=1$, we have that $\|yb+\beta\|_{\infty}\leq3\epsilon$.
Furthermore,
\begin{equation*}
  c\int_{\bar{\mathcal{I}}}(yb_{i}+\beta_{i})b_{i}\dd i\geq c\int_{\bar{\mathcal{I}}}\beta_{i}b_{i}\dd i\geq (1-\eta)\|\beta\|_{\infty},
\end{equation*}
that is, $f-f^\circ=yb+\beta\in\breve{K}^{3\epsilon}_{\mathcal{C}_{\rm b},\mathfrak{b}}(\eta)$, which is in contradiction with (\ref{5.9}) unless $y=0$,
as $\0\notin\Int_{\mathcal{C}_{\rm b}}K_{\mathcal{C}_{\rm b},\mathfrak{b}}(\eta)$. In other words, we must have
$\mathfrak{S}^{(1)}_{f^{\circ},b}(f)=\{0\}$ (whence $\phi^{(1)}_{f^{\circ},b}(f)=0$) for all $f\in D$, proving the claim.
In the same vein, we can verify the constraint in (\ref{5.10}) by arguing that $V^{(2)}_{f^{\circ},b}=0$.
This proves that $\pd_{\mathcal{C}_{\rm b}}D\subset\pd_{\mathcal{C}_{\rm b},\rm s}D$.

\textbf{Step 2.} To prove the other direction, let $f^{\circ}\in\pd_{\mathcal{C}_{\rm b},\rm s}D$. By the definitional condition (\ref{5.13}), there exist $a\in\mathcal{C}_{\rm b}(\bar{\mathcal{I}};\mathbb{R})$, $b\in\pd_{\mathcal{C}_{\rm b}}\overline{B}_{\mathcal{C}_{\rm b}}(1)$, and $k\in\mathbb{Z}_{++}$ such that $V^{(k)}_{a,b}=V^{(k+1)}_{a,b}$ and $\phi^{(k)}_{a,b}(f^{\circ})=V^{(k)}_{a,b}:=y^{\circ}\in[0,k\epsilon]$. Note here that the orient $\mathfrak{b}$ of the cone ${K}_{\mathfrak{b}}(\eta)$ is the Radon measure having the Radon--Nikodym derivative $cb$. Again, $\mathfrak{S}^{(k)}_{a,b}(f^{\circ})$ is a compact subset of $\mathbb{R}_{+}$. Furthermore, following the proof of Theorem \ref{thm:1}, condition (\ref{5.9}) for the boundary $\pd_{\mathcal{C}_{\rm b}}D$ also holds if the (smaller) radius $\epsilon/2$ replaces $r=3\epsilon$.

By way of contradiction, suppose that $f^{\circ}\notin\pd_{\mathcal{C}_{\rm b}}D$. This means that for any $\nu\in\pd_{\mathcal{M}}\overline{B}_{\mathcal{M}}(1)$, there exists $\hat{f}(\nu)\in D$ such that $\hat{f}(\nu)\in f^{\circ}+\breve{K}^{\epsilon/2}_{\mathcal{C}_{\rm b}, \nu}(\eta)$, or equivalently, that $\hat{f}(\nu)=f^{\circ}+\hat{\beta}$ for some $\hat{\beta}\in\breve{K}^{\epsilon/2}_{\mathcal{C}_{\rm b}, \nu}(\eta)$. In what follows we take $\nu=\mathfrak{b}$, where $\mathfrak{b}$ is absolutely continuous with (continuous) Radon--Nikodym derivative $cb$.
From the definition of $\mathfrak{S}^{(k)}_{a,b}(f^{\circ})$ in (\ref{5.10}), we set $\beta^{\circ}:=f^{\circ}-a-y^{\circ}b\in{K}^\epsilon_{\mathcal{C}_{\rm b},\mathfrak{b}}(\eta)$ and assume without loss of generality that $f^{\circ}-a-y^{\circ}b\in{K}^{\epsilon/2}_{\mathcal{C}_{\rm b},\mathfrak{b}}(\eta)$. Then, it follows that $\hat{f}(\mathfrak{b})-a-y^{\circ}b=\hat{\beta}+\beta^{\circ}\in\breve{K}^\epsilon_{\mathcal{C}_{\rm b}, \mathfrak{b}}(\eta)$. Since we know that $\Int_{\mathcal{C}_{\rm b}}{K}_{\mathfrak{b}}(\eta)\neq\emptyset$, there exists $\varepsilon\in(0,\epsilon)$ such that $\hat{\beta}+\beta^{\circ}-\varepsilon b\in\breve{K}^\epsilon_{\mathcal{C}_{\rm b}, \mathfrak{b}}(\eta)$. Hence, by defining $\hat{y}:=y^{\circ}+\varepsilon$, we have that $\hat{f}(\mathfrak{b})-a-\hat{y}b\in\breve{K}^\epsilon_{\mathcal{C}_{\rm b}, \mathfrak{b}}(\eta)$. To finish the claim, we follow step 2 in the proof of Theorem \ref{thm:1}: Based on (\ref{5.13}), if $\hat{y}\leq k\epsilon$, then $\hat{y}\in\mathfrak{S}^{(k)}_{a,b}(\hat{f})$, but $\hat{y}>y^{\circ}$, a contradiction with the fact that $V^{(k)}_{a,b}=y^{\circ}$; if $\hat{y}\in(k\epsilon,(k+1)\epsilon)$, then with $\hat{f}(\mathfrak{b})\in D$, we have that $V^{(k+1)}_{a,b}\geq\hat{y}>y^{\circ}=V^{(k)}_{a,b}$, with $V^{(k+1)}_{a,b}\neq V^{(k)}_{a,b}$ -- also a contradiction. Therefore, we obtain that $\pd_{\mathcal{C}_{\rm b}}D\supset\pd_{\mathcal{C}_{\rm b},\rm s}D$, and along with step 1, that $\pd_{\mathcal{C}_{\rm b}}D=\pd_{\mathcal{C}_{\rm b},\rm s}D$, as required.
\end{proof}


Again, the results of Theorem \ref{thm:4} elucidate a dual representation for general nonconvex boundaries in infinite-dimensional spaces satisfying the exterior cone conditions (\ref{5.4}) and (\ref{5.9}), extending related nonconvex duality results in single-criterion optimization (
\cite{AG02}).

In a manner similar to Section \ref{sec:4.2}, we can reformulate the scalarization problems in (\ref{5.8}) and (\ref{5.12}) (with (\ref{5.5})) for computing the associated implied boundaries (\ref{5.7}) and (\ref{5.13}), respectively. Results are presented in the below theorem.

\begin{theorem}\label{thm:5}
Consider the setting of Theorem \ref{thm:4}, with $\epsilon=r/3>0$ and $\eta\in(0,1]$. When $\mathcal{X}=\mathbb{L}^{2}(\bar{\mathcal{I}};\mathbb{R})$, the value function in (\ref{5.8}) can be written as
\begin{equation}\label{5.16}
  V^{(k)}_{a,b}=\sup_{f\in D}\phi^{(k)}_{a,b}(f)=\min\Big\{\Big(\sup_{f\in D}\mathcal{H}^{\eta}_{a,b}(f)\Big)^{+},k\epsilon\Big\},
\end{equation}
where
\begin{equation}\label{5.17}
  \mathcal{H}^{\eta}_{a,b}(f)=\langle f-a,b\rangle_{\bar{\mathcal{I}}}-\sqrt{\frac{(1-\eta)^{2}(\|f-a\|^{2}_{2}-\langle f-a,b\rangle^{2}_{\bar{\mathcal{I}}})}{\eta(2-\eta)}}.
\end{equation}

When $\mathcal{X}=\mathcal{C}_{\rm b}(\bar{\mathcal{I}};\mathbb{R})$, the value function in (\ref{5.13}) can be written as
\begin{equation}\label{5.18}
  V^{(k)}_{a,b}=\sup_{f\in D}\phi^{(k)}_{a,b}(f)=\min\Big\{\Big(\sup_{f\in D}\tilde{\mathcal{H}}^{\eta}_{a,b}(f)\Big)^{+},k\epsilon\Big\},
\end{equation}
where $\tilde{\mathcal{H}}^{\eta}_{a,b}(f)$ is the unique real solution to the transcendental equation
\begin{equation}\label{5.19}
  \langle f-a-\tilde{\mathcal{H}}^{\eta}_{a,b}(f)b,\mathfrak{b}\rangle_{\bar{\mathcal{I}}} =(1-\eta)\sup_{i\in\bar{\mathcal{I}}}|f_{i}-a_{i}-\tilde{\mathcal{H}}^{\eta}_{a,b}(f)b_{i}|.
\end{equation}
\end{theorem}

\begin{proof}
First, the conclusion of Lemma \ref{lem:1} holds directly with $\mathbb{R}^{m}$ replaced by $\mathcal{X}=\mathbb{L}^{2}(\bar{\mathcal{I}};\mathbb{R})$, due to its reflexivity, using the norm $\|\cdot\|_{2}$ and the inner product $\langle\cdot,\cdot\rangle_{\bar{\mathcal{I}}}$. Then, we are led to solve the conditional quadratic equation
\begin{equation*}
  \langle f-a-yb,b\rangle_{\bar{\mathcal{I}}}=(1-\eta)\|f-a-yb\|_{2},\quad y\in\mathbb{R},
\end{equation*}
similar to in the proof of Theorem \ref{thm:3}, whose solution is uniquely determined as $y=\mathcal{H}^{\eta}_{a,b}(f)\in\mathbb{R}$ as in (\ref{5.17}), which projected into the range $[0,k\epsilon]\ni y$ with $k\in\mathbb{Z}_{++}$ given yields (\ref{5.16}).

Second, if $\mathcal{X}=\mathcal{C}_{\rm b}(\bar{\mathcal{I}};\mathbb{R})$ (which is nonreflexive), by way of contradiction again, suppose that, given $k\in\mathbb{Z}_{++}$, $y^{\ast}\in(0,k\epsilon)$ and $f-a-y^{\ast}b\in\Int_{\mathcal{C}_{\rm b}}K_{\mathcal{C}_{\rm b},\mathfrak{b}}(\eta)$. Then, by (\ref{5.15}), $(1-\eta)\|f-a-y^{\ast}b\|_{\infty}\leq c\int_{\bar{\mathcal{I}}}(f_{i}-a_{i}-y^{\ast}b_{i})b_{i}\dd i+\varepsilon_{1}$ for some $\varepsilon_{1}>0$. Then, the arguments in (\ref{4.1}) can be reproduced as
\begin{align*}
  (1-\eta)\|f-a-\hat{y}b\|_{\infty}&\leq(1-\eta)\|f-a-y^{\ast}b\|_{\infty}+\varepsilon_{2}(1-\eta) \\
  &\leq(1-\eta)\|f-a-y^{\ast}b\|_{\infty}+\varepsilon_{1}-\varepsilon_{2} \\
  &\leq c\int_{\bar{\mathcal{I}}}(f_{i}-a_{i}-y^{\ast}b_{i})b_{i}\dd i-\varepsilon_{2} \\
  &=c\int_{\bar{\mathcal{I}}}(f_{i}-a_{i}-\hat{y}b_{i})b_{i}\dd i,
\end{align*}
where $\varepsilon_{2}\in(0,\varepsilon_{1}/(2-\eta))$ exists such that $\hat{y}=y^{\ast}+\varepsilon_{2}\in(0,k\epsilon]$, and the last equality is possible thanks to condition (\ref{5.11}). This contradicts the maximality of $y^{\ast}$ from (\ref{5.12}). Therefore, $f-a-y^{\ast}b\in\pd_{\mathcal{C}_{\rm b}}K_{\mathcal{C}_{\rm b},\mathfrak{b}}(\eta)=\{\beta\in\mathcal{C}_{\rm b}(\bar{\mathcal{I}};\mathbb{R}):\;(1-\eta)\|\beta\|_{\infty}=\langle\beta,\mathfrak{b}\rangle_{\bar{\mathcal{I}}}\}$ as well.

The boundary condition herein gives rise to the following equation:
\begin{equation}\label{5.20}
  c\int_{\bar{\mathcal{I}}}(f_{i}-a_{i}-\hat{y}b_{i})b_{i}\dd i=(1-\eta)\sup_{i\in\bar{\mathcal{I}}}|f_{i}-a_{i}-yb_{i}|,\quad y\in\mathbb{R},
\end{equation}
which is generally transcendental due to the supremum operator, and in which $c\int_{\bar{\mathcal{I}}}(f_{i}-a_{i}-yb_{i})b_{i}\dd i=\langle f-a-yb,\mathfrak{b}\rangle_{\bar{\mathcal{I}}}$ for the scaling factor $c$ determined based on $\bar{\mathcal{I}}$.

While it does not seem possible to solve (\ref{5.20}) explicitly, we observe that its right side forms a Lipschitz-continuous function in $y\in\mathbb{R}$, with Lipschitz constant equal to $(1-\eta)\|b\|_{\infty}=1-\eta\in[0,1)$. The left side is a linear function (also in $y\in\mathbb{R}$) that is downward-sloping with slope magnitude $c\int_{\bar{\mathcal{I}}}b^{2}_{i}\dd i=1$, recalling (\ref{5.11}). Thus, (\ref{5.20}) must have a unique real solution, determined by (\ref{5.19}). By projection again, we obtain (\ref{5.18}), completing the proof.
\end{proof}

\smallskip

For the case $\mathcal{X}=\mathbb{L}^{2}(\bar{\mathcal{I}};\mathbb{R})$, we briefly discuss potential solution methods for (\ref{5.16}), considering $f:\mathcal{A}\mapsto\mathcal{X}$ an infinite-dimensional optimization objective function, along with the understanding that $D=\cl_{\mathcal{X}}\{f(x):\;x\in\mathcal{A}\}$. We adapt the procedures outlined in Section \ref{sec:4.2}.

By setting $\theta=(1-\eta)/\sqrt{\eta(2-\eta)}\geq0$ as in (\ref{4.7}) and choosing a sequence $\{b^{\perp,(j)}:\;j\in\mathbb{Z}_{++}\}\cup\{b\}\subset\overline{B}_{\mathbb{L}^{2}}(1)$ forming an orthonormal basis for $\mathbb{L}^{2}(\bar{\mathcal{I}};\mathbb{R})$, the relation (\ref{4.8}) becomes
\begin{equation}\label{5.21}
  F_{a,b}=\sup_{x\in\mathcal{A}}\Bigg\{\langle f(x)-a,b\rangle_{\bar{\mathcal{I}}}-\theta\sqrt{\sum^{\infty}_{j=1}\langle f(x)-a,b^{\perp,(j)}\rangle^{2}_{\bar{\mathcal{I}}}}\Bigg\};
\end{equation}
the validity of this transformation is due to (\ref{5.17}). Then, (\ref{5.21}) can be solved in two stages following the same idea as for (\ref{5.17}). It will necessarily lead to an (infinity) sequence of (linear) constraints to consider. More specifically, suppose that we solve the resultant unconstrained problem
\begin{equation}\label{5.22}
  G_{a,b}(l):=\sup_{x\in\mathcal{A}}\Bigg\{\langle f(x),b\rangle_{\bar{\mathcal{I}}}+\sum^{\infty}_{j=1}\lambda_{l_{j}}(\langle f(x),b^{\perp,(j)}\rangle_{\bar{\mathcal{I}}}-\langle a,b^{\perp,(j)}\rangle_{\bar{\mathcal{I}}}-l_{j})\Bigg\},
\end{equation}
where $l_{j},\lambda_{l_{j}}\in\mathbb{R}$, for every $j\in\mathbb{Z}_{++}$. Then, the Cauchy--Schwarz inequality immediately shows that a necessary and sufficient condition for the sum in (\ref{5.22}) to converge is that the associated sequences $l=\{l_{j}:\;j\in\mathbb{Z}_{++}\}$ and $\lambda_{l}=\{\lambda_{l_{j}}:\;j\in\mathbb{Z}_{++}\}$ belong to the space $\ell^{2}(\mathbb{R})$ of square-summable real sequences, which is also a Hilbert space, equipped with the norm $\|\cdot\|_{\ell^{2}}$. In this case, the required value function (\ref{5.21}) can be represented for each $k\in\mathbb{Z}_{++}$ as $V^{(k)}_{a,b}=\min\big\{F^{+}_{a,b},k\epsilon\big\}$, where
\begin{equation*}
  F_{a,b}=\sup_{l\in\ell^{2}(\mathbb{R})}\big\{G_{a,b}(l)-\langle a,b\rangle_{\bar{\mathcal{I}}}-\theta\|l\|_{\ell^{2}}\big\}.
\end{equation*}
This enables computing (\ref{5.16}) in terms of $V^{(k)}_{a,b}=\min\big\{F^{+}_{a,b},k\epsilon\big\}$.

\medskip

\section{Discussion on related applications}\label{sec:6}

In this section, we discuss several key applications where the proposed boundary characterization approach plays a significant role, building on the discussion in Section \ref{sec:1}. In the first part, we examine existing stochastic optimization and game-theoretic settings where the boundary of sets is either a primary objective or a subject of established interest and then use the characterization to derive compact representation formulas for these boundaries. In the second part, we establish new connections to machine learning problems (pertaining to the spatial structures of loss landscapes) and demonstrate how the characterization gives useful insights for this domain.\footnote{All of these applications could be explored further in their respective settings and are not exhaustive -- many other fields, such as trajectory planning in robotics and energy landscape analysis in materials science, may also benefit from this characterization.}

\smallskip


\subsection{Direct applications: Multiplicity in optimization and games}\label{sec:6.1}

First, consider the motivating problem at the beginning of Section \ref{sec:5}. In this scenario, the control variable in (\ref{4.5}), in a generalized sense, may be understood as the strategy profile $x=\alpha^{\star}\equiv(\alpha^{\star}_{1},\dots,\alpha^{\star}_{N})$, which takes values in the set
\begin{equation*}
  \mathcal{A}=\Bigg\{\alpha^{\star}\in\prod^{N}_{i=1}A_{i}:\;J_{i}(\alpha^{\star}_{i},\alpha^{\star}_{\neg i})=\sup_{\alpha_{i}\in A_{i}}J_{i}(\alpha_{i},\alpha^{\star}_{\neg i}),\;i\in\mathbb{Z}\cap[1,N]\Bigg\},
\end{equation*}
containing all equilibrium strategies (which itself is to be obtained using a fixed-point argument). Interest is in characterizing the set $D=\cl\{J(\alpha^{\star}):\alpha^{\star}\in\mathcal{A}\}\subset\mathbb{R}^{N}$ of game values,\footnote{To be clear, $D$ represents the game's \textsl{raw} set values, used here for illustrative convenience, in contrast to the actual set value which is derived from the limit of the values of $\epsilon$-equilibria; we refer to \cite[\S2]{IZ24a} for technical details.} where $J(\alpha^{\star})$ is the same as $(J_{1}(\alpha^{\star}_{1}),\dots,J_{N}(\alpha^{\star}_{N}))$.

To this end, assume that $r\to\infty$ is allowed in condition (\ref{3.1}). Then, by Theorem \ref{thm:1} and Theorem \ref{thm:3}, the boundary of $D$ can then be expressed as
\begin{equation}\label{6.1.1}
  \pd D=\cl\Big\{J(\alpha^{\ast}):\;\alpha^{\ast}\in\underset{\alpha^{\star}\in\mathcal{A}}{\arg\max}H^{\eta}_{a,b}(J(\alpha^{\star})),\;\exists a\in\mathbb{R}^{N},\;b\in\pd\overline{B}(1)\text{ s.t. }H^{\eta}_{a,b}(J(\alpha^{\ast}))\geq0\Big\},
\end{equation}
where, again, $H^{\eta}_{a,b}$ is given by (\ref{4.3}). Note that the set $D$ may very well be countable or even finite, and we simply have $\pd D=D$ (in $\mathbb{R}^{N}$), and the characterization of the boundary means the whole set itself. Thus, the representation (\ref{6.1.1}) is particularly useful when $D$ is known or reasonably suspected to have continuum cardinality with (multiple) connectedness.

In mean-field game scenarios where there are infinitely many players, similar representations are immediately available based on the infinite-dimensional setting discussed in Section \ref{sec:5}, depending on the emphasis on integrability or continuity.

\smallskip

Second, in multi-criteria stochastic optimization problems, a useful tool for dealing with potential time inconsistency issues, i.e., the temporal sub-optimality of previously optimal controls, consists in characterizing the dynamics of the (multi-criteria) objective function; see, again, 
\cite{IZ24b} for details.

More precisely, consider a decision-maker with a finite number $m\geq2$ of objective functions (or a group of $m$ decision-makers each with a single objective function). Given a $d$-dimensional Brownian motion in $[0,T]$, at each initial state $(t,z)\in[0,T)\times\mathbb{R}^{d}$, the decision-maker chooses a control $\alpha$ from the admissibility set $\mathcal{A}_{t}$ consisting of progressively measurable processes with values in a certain Euclidean space. His stochastic optimization is characterized with the following controlled dynamics:
\begin{align*}
  X^{t,z,\alpha}_{v}&=z+\int^{v}_{t}\mu(s,X^{t,z,\alpha}_{s},\alpha_{s})\dd s+\int^{v}_{t}\sigma(s,X^{t,z,\alpha}_{s},\alpha_{s})\dd W_{s},\\
  Y^{t,z,\alpha}_{v}&=U(X^{t,z,\alpha}_{T})+\int^{T}_{v}u(s,X^{t,z,\alpha}_{s},Y^{t,z,\alpha}_{s},Z^{t,z,\alpha}_{s},\alpha_{s}\big)\dd s-\int^{T}_{v}Z^{t,z,\alpha}_{s}\dd W_{s},
\end{align*}
where the processes $X$ (state), $Y$ (objective), and $Z$ (auxiliary) are valued in $\mathbb{R}^{d}$, $\mathbb{R}^{m}$, and $\mathbb{R}^{m\times d}$, respectively, and $\mu$ and $\sigma$ are respectively the drift and volatility coefficients of appropriate dimensions, and $u$ and $U$ are his objective functions with values in $\mathbb{R}^{m}$.

In this context, attention is drawn to the set value function $D\equiv D^{t,z}=\cl\{Y^{t,z,\alpha}_{t}:\alpha\in\mathcal{A}_{t}\}$. By Theorem \ref{thm:3}, its boundary has the representation
\begin{align}\label{6.2.1}
  \pd D^{t,z}&=\cl\Big\{Y^{t,z,\alpha^{\ast}}_{t}:\;\alpha^{\ast}\in\underset{\alpha\in\mathcal{A}_{t}}{\arg\max}H^{\eta}_{a,b}(Y^{t,z,\alpha}_{t}),\; \exists a\in\mathbb{R}^{m},\;b\in\pd\overline{B}(1)\text{ s.t. }H^{\eta}_{a,b}(Y^{t,z,\alpha^{\ast}}_{t})\geq0\Big\}, \nonumber\\
  &\quad(t,z)\in[0,T]\times\mathbb{R}^{m},
\end{align}
assuming that for every $(t,z)$, the set value $D^{t,z}$ satisfies (\ref{3.1}) with $r\to\infty$. Clearly, the maximization in (\ref{6.2.1}) is standard stochastic optimization, making it well fit into the set-valued dynamic programming principle (
\cite[\text{Thm.} 3.1]{IZ24b}).

In case there are infinitely many criteria, depending on whether such criteria satisfy certain continuity or integrability property -- with respect to the criterion label or index (see below), (\ref{6.2.1}) readily generalizes to the function-valued parameters $a$ and $b$, as was discussed in Section \ref{sec:5}.

\smallskip

Third, for multi-criteria optimization in infinite stochastic dimensions (
\cite{LM20} and 
\cite{X24}), the boundary of interest is that of some set of parameters governing the time-invariant dimensionality, which represents either Knightian uncertainty or preference indecisiveness.

A notable example is the consumption--investment problem in markets with multiple goods (count $n\geq2$). More specifically, consider a representative agent who is indecisive among goods; her preferences are incomplete and admit the so-called ``multi-utility representation theorem'' (see 
\cite{EO11}). The multi-utility is then associated with a collection of classical utility functions $\{u_{i}(c):i\in\mathcal{I}\}$, where $c\in\mathbb{R}^{n}$, $n\geq2$, is an $n$-vector of good-specific consumption and $\mathcal{I}\subset\mathbb{R}^{m}$, $m\geq1$, is a utility index set containing $m$ utility parameters. The stochastic dimensionality feature appears if the agent's preferences can change over time -- exogenously or otherwise -- conforming to certain set-valued dynamics; see 
\cite[\S6]{X25} for a general description. In this scenario, the boundary $\pd\mathcal{I}$ plays a crucial role in computing the optimal investment policies based on set-valued Malliavin calculus (
\cite[\text{Thm.} 4]{X24}).

For instance, a possibly nonconvex utility index set $\mathcal{I}$ that encodes influences of market characteristics on the agent's preference variations over a time interval $[0,T]$ can be taken as the (exogenous) Aumann integral
\begin{equation*}
  \mathcal{I}_{t}=\int^{t}_{0}X_{s}\dd s,\quad t\in[0,T],
\end{equation*}
where $X$ is a non-anticipating function with compact (not necessarily convex) values; for simplicity we still assume that $r\to\infty$ is permissible in condition (\ref{3.1}). Then, the boundary of the utility index set in $\mathbb{R}^{m}$ at any time point $t\in[0,T]$ can be represented as ($\PP$-a.s.)
\begin{equation*}
  \pd\mathcal{I}_{t}=\cl\bigg\{\int^{t}_{0}x_{s}\dd s:\;x^{\ast}\in\underset{x\in X_{[0,t]}}{\arg\max}H^{\eta}_{a,b}\bigg(\int^{t}_{0}x_{s}\dd s\bigg),\;\exists a\in\mathbb{R}^{m},\;b\in\pd\overline{B}(1)\text{ s.t. }H^{\eta}_{a,b}\bigg(\int^{t}_{0}x^{\ast}_{s}\dd s\bigg)\geq0\bigg\}.
\end{equation*}
Since $\mathcal{A}=X_{[0,t]}$ is path-dependent, as is $\mathcal{I}_{t}$, the maximization should be understood as deterministic optimization, significantly lifting simulation burdens (
\cite[\S6]{X24}).

To some extent, the same problem may be analogized to general incomplete-market equilibrium models with a large number of heterogeneous agents labeled by $m$ traits, which have a time-dependent distribution. Then, the (random) index set $\mathcal{I}\subset\mathbb{R}^{m}$ may be interpreted as the (dynamic) collection of labels for such agents in a certain state; see 
\cite[\S2]{L25} for a comprehensive viewpoint. In this case, the boundary (surface) $\pd\mathcal{I}$ is sufficient to determine the agents' law of motion -- in a lower dimension.

\smallskip

\subsection{New connections: Constrained reinforcement learning}\label{sec:6.2}

In reinforcement learning (RL), the space of feasible policies, namely the set of all possible mappings from predetermined states to actions, tend to be highly nonconvex. In particular, in continuous control problems or in cases where policies are represented by deep neural networks, the space of admissible policies (which is straightly related to the loss landscape) is oftentimes nonlinear and constrained by dynamics that are nonconvex in nature.

Plainly speaking, the set of all admissible policies $\mathcal{P}$ represents policies that satisfy the constraints of the environment.\footnote{Some standard examples include avoiding obstacles in a robotic control task, staying within a certain financial budget, or achieving certain safety conditions.} Of special interest in constrained RL is the boundary of this set, where the goal is to maximize some reward while remaining within the admissible region (domain).

For instance, consider deterministic parameterized policies $\pi_{\theta}$ that maps states $s\in\mathcal{S}$ into actions $\alpha\in\mathcal{A}$, where $\mathcal{S}$ is the state space and $\mathcal{A}$ is the action space, with $\theta$ being parameters (e.g., from a neural network architecture) taking values in some parameter space $\Theta$. Then, the admissible policy space (or feasible region) is given by
\begin{equation*}
  \mathcal{P}=\cl\{(\pi_{\theta}:\mathcal{S}\mapsto\mathcal{A}):\;\theta\in\mathrm{\Theta}_{\text{constr}}\},
\end{equation*}
where $\mathrm{\Theta}_{\text{constr}}$ contains constraints on the parameters $\theta$, generating a highly nonconvex feasible region of policies (as a neuro-manifold (
\cite{C20})). Finding policies on the boundary of this space, which contains key information for the optimality of objectives, is a rather difficult problem, as traditional methods like policy gradient or Q-learning may struggle to handle the nonconvexity of $\mathcal{P}$; we refer to the discussion in 
\cite{HXY23}.

Given this highly nonconvex nature, we do not assume that $r\to\infty$ in condition (\ref{3.1}) for practicality. Based on Theorem \ref{thm:1} and Theorem \ref{thm:3}, the boundary of the policy space (or the neuro-manifold) can then be computed as
\begin{align}\label{6.3.1}
  \pd\mathcal{P}&=\cl_{\#}\bigg\{\pi_{\theta^{\ast}}:\; \theta^{\ast}\in\underset{\theta\in\Theta_{\text{constr}}}{\arg\max}\min\bigg\{(\mathcal{H}^{\eta}_{a,b})^{+}(\pi_{\theta}),\frac{kr}{3}\bigg\}, \;\exists a\in\#,\;b\in\overline{B}_{\#'}(1),\;k\in\mathbb{Z}\cap[1,k^{\ast}] \nonumber\\
  &\quad\text{ s.t. }\min\bigg\{(\mathcal{H}^{\eta}_{a,b})^{+}(\pi_{\theta^{\ast}}),\frac{(k+1)r}{3}\bigg\} =\min\bigg\{(\mathcal{H}^{\eta}_{a,b})^{+}(\pi_{\theta^{\ast}}),\frac{kr}{3}\bigg\}\bigg\} \nonumber\\
  &=\cl_{\#}\bigg\{\pi_{\theta^{\ast}}:\; \theta^{\ast}\in\underset{\theta\in\Theta_{\text{constr}}}{\arg\max}\mathcal{H}^{\eta}_{a,b}(\pi_{\theta}),\;\exists a\in\#,\;b\in\overline{B}_{\#'}(1),\;k\in\mathbb{Z}\cap[1,k^{\ast}] \nonumber\\
  &\quad\text{ s.t. }H^{\eta}_{a,b}(\pi_{\theta^{\ast}})\in\bigg[0,\frac{kr}{3}\bigg]\bigg\},
\end{align}
where $k^{\ast}$ is some (verifiably) sufficiently large number such that $\sup_{\theta\in\Theta_{\text{constr}}}(H^{\eta}_{a,b})^{+}(\pi_{\theta})\leq k^{\ast}r/3$ for all parameters $a$ and $b$. Depending on the structure of $\mathcal{S}$ and $\mathcal{A}$, the dimensionality of $a$ and $b$ can be either finite or infinite, hence the use of the symbol $\#$ as a placeholder (with topological dual $\#'$). If the constraints imposed on $\theta$ (or indirectly on $\pi$) are all hard constraints, then the optimal policies should generally lie on the boundary $\pd\mathcal{P}$, for which using the representation (\ref{6.3.1}) greatly facilitates the subsequent training process.\footnote{In RL tasks with safety constraints (such as safe autonomous driving), the feasible region has complex boundary geometry (e.g., disconnectedness) due to intricate state dynamics. Just as an example, consider if the policy has to ensure the car stays on the road while avoiding pedestrians, and then the policy space will be limited by dynamic constraints that are highly nonlinear and nonconvex.}

\medskip

\section{Concluding remarks}\label{sec:7}

We have proposed a general method for characterizing the boundaries of generally nonconvex domains (Theorem \ref{thm:1}, Theorem \ref{thm:4}). Considering the boundary search process a collection of multi-criteria optimization problems, the characterization makes use of indefinite, possibly bounded spherical cones (Definition \ref{def:1}) to deal with highly nonconvex regions, which are only required to satisfy a natural exterior cone condition (Assumption \ref{as:1}, Assumption \ref{as:2}, or Assumption \ref{as:3}). Alternatively, the characterization can be seen as a general dual representation of nonconvex sets (or boundaries), significantly extending existing results on convex sets using support functions, as well as advancements in nonconvex duality theory. With additional boundedness conditions, the dimensionality of the parametric range can be progressively reduced, allowing to further simplify the characterization (Theorem \ref{thm:2} and Corollary \ref{cor:1}).

Employing nonlinear scalarization functions of Pascoletti--Serafini type, we have shown that the multi-criteria optimization problem, for fixed parametrization, is equivalent to a single-criterion constrained nonconvex optimization problem. With this particular formulation, an explicit problem can be established and easily solved following standard procedures in (deterministic or stochastic) optimization (Theorem \ref{thm:3} and Theorem \ref{thm:5}).

Further research could examine rare situations where the exterior cone condition is violated -- locally. In this case, an approximating sequence of spherical cones with radius $r$ and sharpness $\eta$ both tending to 0 may be used for a theoretical justification. Additionally, while we have demonstrated the numerical implementation of the method discussed in Section \ref{sec:4.3}, further efforts can be directed toward developing efficient algorithms for solving the resultant problem in (\ref{4.2}) with (\ref{4.3}), particularly in stochastic settings, potentially leveraging deep learning techniques. Last but not least, each application discussed in Section \ref{sec:6} presents a problem of independent interest, also warranting the design of tailored algorithms.

\end{document}